\documentclass[a4paper,draft]{amsart}
\usepackage{graphicx}
\usepackage{graphics}
\usepackage{amsmath}
\usepackage{amscd,enumerate}
\usepackage{latexsym}
\usepackage{amssymb}
\usepackage[all]{xy}
\usepackage{psfrag}
\usepackage[ansinew]{inputenc}
\usepackage[T1]{fontenc}
\usepackage{lmodern}
\xyoption{matrix} \xyoption{arrow}

\newcommand{\id}{\operatorname{id}\nolimits}

\newcommand{\A}{\operatorname{\mathcal A}\nolimits}

\newcommand{\D}{\operatorname{\mathcal D}\nolimits}
\newcommand{\E}{\operatorname{\mathcal E}\nolimits}
\newcommand{\X}{\operatorname{\mathcal X}\nolimits}
\newcommand{\Y}{\operatorname{\mathcal Y}\nolimits}

\newcommand{\Z}{\operatorname{\mathcal Z}\nolimits}

\newcommand{\ord}{\operatorname{ord}\nolimits}

\renewcommand{\L}{\Lambda}
\newcommand{\G}{\Gamma}
\newcommand{\Gb}{\overline{\Gamma}}
\newcommand{\Q}{{\mathcal Q}}
\newcommand{\Qg}{{\mathcal Q}_{\Gamma}}
\newcommand{\Qgb}{{\mathcal Q}_{\overline{\Gamma}}}
\newcommand{\barQg}{\overline{{\mathcal Q}_{\Gamma}}}
\newcommand{\Ig}{I_{\Gamma}}

\newcommand{\barIg}{\overline{I_{\Gamma}}}
\newcommand{\Ag}{\A_{\Gamma}}

\newcommand{\oQ}{\overline\Q}
\newcommand{\oL}{\overline \Lambda}

\newcommand{\mo}{{\mathfrak o}}
\newcommand{\om}{\overline{m}}
\newcommand{\omo}{\overline{\mathfrak o}}
\newcommand{\oq}{\overline{q}}

\newcommand{\Mod}{\operatorname{Mod}\nolimits}
\newcommand{\val}{\operatorname{val}\nolimits}

\newtheorem{lemma}{Lemma}[section]
\newtheorem{proposition}[lemma]{Proposition}
\newtheorem{corollary}[lemma]{Corollary}
\newtheorem{theorem}[lemma]{Theorem}

\newtheorem{definition}[lemma]{Definition}
\newenvironment{dfn}{\begin{definition} \rm}{\end{definition}}

\newtheorem{example}[lemma]{Example}
\newenvironment{ex}{\begin{example} \rm}{\end{example}}

\begin{document}

\topmargin 0cm
\oddsidemargin 0.5cm
\evensidemargin 0.5cm
\baselineskip=15pt

\title[Group actions and coverings of Brauer graph algebras]{Group actions and coverings of Brauer graph algebras}

\author[Green]{Edward L. Green}
\address{Edward L. Green\\
Department of Mathematics \\
Virginia Tech\\
Blacksburg, VA  24061-0123\\
USA}
\email{green@math.vt.edu}

\author[Schroll]{Sibylle Schroll}
\address{Sibylle Schroll\\
Department of Mathematics \\
University of Leicester \\
University Road \\
Leicester LE1 7RH \\
United Kingdom}
\email{ss489@le.ac.uk}

\author[Snashall]{Nicole Snashall}
\address{Nicole Snashall\\
Department of Mathematics \\
University of Leicester \\
University Road \\
Leicester LE1 7RH \\
United Kingdom}
\email{njs5@mcs.le.ac.uk}

\subjclass[2010]{Primary 05E18, 16G20; Secondary 14E20, 16W50, 58E40}
\keywords{Group actions, Brauer graph algebras, coverings of algebras, coverings of graphs, group gradings}

\begin{abstract}
We develop a theory of group actions and coverings on Brauer graphs that parallels the theory of group actions and
coverings  of algebras.  In particular, we show that any Brauer graph can be
covered by a tower of coverings of  Brauer graphs such that the topmost covering has multiplicity function identically
one, no loops, and no multiple edges. Furthermore, we classify the coverings
of Brauer graph algebras that are again Brauer graph algebras.
\end{abstract}

\thanks{This work was supported by the Leverhulme Trust through an
 Early Career Fellowship for the second author.}

\maketitle

\section*{Introduction}

In this paper we introduce the theory of group actions and coverings on Brauer
graphs that parallels the theory of group actions and
coverings  of algebras.  In particular,  we show that any Brauer graph can be
covered by a tower of coverings, the topmost of which is a Brauer graph with no exceptional vertices,
 no loops, and no multiple edges (Theorem~\ref{thm:final}). This allows many homological questions
 related to Brauer graph algebras to be simplified by considering
 Brauer graphs with no exceptional vertices, no loops and no multiple edges.
Specifically, we know from covering theory that, for a finite group $G$, the category of $G$-graded modules over a Brauer graph algebra is equivalent to the module category of the covering algebra associated to the group $G$. In a subsequent paper (\cite{GSST}), we use this fact and Theorem~\ref{thm:final} to compute a minimal set of generators for the Ext algebra of any Brauer graph algebra.
In this current paper we also classify the coverings
of Brauer graph algebras that are again Brauer graph algebras
(Theorem~\ref{thm:covering}). One reason this result is of interest, is
that it is difficult, in general, to determine if a covering comes from
a Brauer graph simply from knowledge of the quiver and relations.

Since their introduction by Riedtmann in \cite{Riedtmann80}, coverings of algebras have been extensively studied
and they have rapidly proven to be of interest in the representation theory of algebras (see, for example,
\cite{BongartzGabriel81, GordonGreen82, Green83, GHS08, SY05} and references within). In the case of self-injective algebras of finite or tame type - including Brauer graph algebras - this is also well demonstrated in the survey article \cite{Skowronski06}. The results in \cite{Green83} and \cite{GHS08},
relating the coverings of algebras given by quivers with relations on one hand, and the group-gradings of algebras on the other hand, lead to an equivalence of categories; namely, given a $G$-grading of a quiver with relations, then the category of finite-dimensional modules of the covering of the algebra is equivalent to the category of finite-dimensional $G$-graded modules.

Brauer graphs are a generalization of Brauer
trees. A Brauer tree is a finite tree, with a cyclic ordering of the
 edges at each vertex
and a multiplicity assigned to exactly one of its vertices, which is called the exceptional vertex.  These trees
were defined by Brauer in \cite{Brauer41} in the study of block
algebras of finite groups with cyclic defect groups. The entire
structure of a block with cyclic defect groups can be read directly
from its
Brauer tree. For example,  in \cite{Janusz69}, Janusz gives a
description of the non-projective, non-simple indecomposable modules
of a block with cyclic defect groups in terms of paths in its Brauer
tree, and in \cite{Green74}, J.A. Green defined ``walking around the
Brauer tree", giving projective resolutions of certain modules of
a block with cyclic defect groups.

It is well known (see for example \cite{Alperin86} or
\cite{Benson98}) that Brauer graph algebras are tame special
biserial self-injective algebras and that those of finite
representation type are the Brauer tree algebras. Furthermore, the
derived equivalence classes of Brauer graph algebras have been
extensively studied, beginning with Brauer tree algebras
(for example, \cite{Rickard89}, and, for stable equivalence classes, see \cite{Reiten77}), generalized Brauer tree
algebras (see \cite{Membrillo-Hernandez97}) and finally Brauer graph
algebras (see \cite{Kauer96, Kauer98, KauerRoggenkamp01, Roggenkamp96}
and \cite{Skowronski06} and the references within).

We briefly summarize the results of the paper.  In Section~\ref{section:notation}, we
begin by recalling the definitions of a Brauer graph and a Brauer graph algebra
together with some essential notation.  In Section~\ref{section:iso}, we define a free Brauer action of a
finite abelian group on a Brauer graph and construct the Brauer orbit graph.  A free Brauer action of a finite abelian group $G$
on a Brauer graph $\G$ induces a free action of $G$ on the quiver of the associated Brauer graph algebra and the section culminates with Theorem~\ref{thm:firstiso}, which relates the Brauer graph algebras
of $\G$ and its orbit graph $\overline{\G}$.

Section~\ref{section:weight} introduces a Brauer weighting on a Brauer graph which is analogous to a weight function on a quiver.  Given a Brauer weighting $W$ to a finite abelian group $G$ on a Brauer graph $\Delta$, we define the Brauer
covering graph $\Delta_W$ and show there is a canonical action of $G$ on
$\Delta_W$.  Theorem~\ref{thm:inducedfreeBraueraction} proves that this action is a free Brauer action and that $\Delta\cong\overline{\Delta_W}$ as Brauer graphs. We apply the results of Section~\ref{section:weight} in Section~\ref{section:q} to Brauer graph algebras.  If $\A_{\Delta}$ (respectively, $\A_{\Delta_{W}}$, $\A_{\overline{\Delta_W}}$) is the Brauer graph algebra of  $\Delta$ (resp., $\Delta_{W}$, $\overline{\Delta_W}$), then Theorem~\ref{thm:algiso} shows that there are $K$-algebra isomorphisms
\[
\A_{\Delta}\cong \A_{\overline{\Delta_W}}\cong\overline{\A_{\Delta_{W}}},
\]
where $\overline{\A_{\Delta_{W}}}$ denotes the orbit algebra of   $\A_{\Delta_{W}}$.

The main result of Section~\ref{section:quot-to-coverings} is Theorem~\ref{thm:orbit2covering}, which shows that if $G$ is a finite
abelian group with free Brauer action on  $\Gamma$, then there is a
Brauer weighting $W$ on $\overline \G$ such that the Brauer graphs $\G$ and $\overline{\G}_W$ are
isomorphic, as are their corresponding Brauer graph algebras.

Applications of the theory presented are given in Section~\ref{section:appl}. The Appendix, Section~\ref{section:app}, provides
a brief survey of covering theory for path algebras and their quotients.

\section{Notation}\label{section:notation}

Let $\Gamma$ be a finite connected graph with at least one edge.  We denote
by $\Gamma_0$ the set of vertices of $\Gamma$ and by $\Gamma_1$ the
set of edges of $\Gamma$. We equip $\Gamma$ with a {\em multiplicity
function} $m \colon \Gamma_0 \rightarrow \mathbb{N} \setminus \{0\}$
and, for each vertex in $\Gamma$, we fix a cyclic ordering of the edges incident with this vertex.
We call such a graph a {\em Brauer graph}. Thus a Brauer graph is a triple $(\Gamma, \mathfrak{o}, m)$
where $\mathfrak{o}$ denotes the cyclic ordering and $m$ the multiplicity function. Note that in the Brauer tree case, the multiplicity function has value $1$ at all but possibly one vertex, called the exceptional vertex.
Although the convention in the literature is to denote a Brauer graph by $\Gamma$, where
the choice of cyclic ordering and multiplicity function are suppressed, we usually write $(\G,\mo,m)$.
In all examples a planar embedding of $\Gamma$ is given and we
choose the cyclic ordering to be the clockwise ordering of the edges
around each vertex.

We say that an edge $j$ in $\G$ is the {\em successor} of the edge
$i$ at the vertex $\mu$ if both $i$ and $j$ are incident with $\mu$ and
edge $j$ directly follows edge $i$ in the cyclic ordering around $\mu$.
For each $\mu \in \Gamma_0$, let $\val(\mu)$
denote the {\em valency} of $\mu$, that is, the number of edges
incident with $\mu$ where we count each loop as two edges.
If $\val(\mu) = 1$ with edge $i$ incident with the vertex $\mu$ then we say that
$i$ is a successor of itself. If $\mu$ is a vertex with $\val(\mu) = 1$ and $m(\mu) = 1$
so that $i$ is the only edge incident with $\mu$ then we call $i$ a {\em truncated edge}
at the vertex $\mu$.

Following \cite{Benson98} and \cite{Kauer98}, we let $K$ be a field and
introduce the Brauer graph algebra of a Brauer graph $\Gamma$.
We associate to $\Gamma$ a quiver $\Qg$ and a
set of relations $\rho_\Gamma$ in the path algebra $K\Qg$, which we call
the {\em Brauer graph relations}.
Let $\Ig$ be the ideal of $K\Qg$ which is generated by the set $\rho_\Gamma$.
We define the {\em Brauer graph algebra} $\Ag$ of $\Gamma$ to be the quotient $\Ag = K\Qg/\Ig$.

If the Brauer graph $\G$ is
$\xymatrix{
\mu \ar@{-}[r] & \nu
}$
with $m(\mu) = m(\nu) = 1$ then
$\Qg$ is
$\xymatrix{
\bullet \ar@(dr,ur)[]_x
}$
and $\rho_\Gamma = \{x^2\}$ so
the Brauer graph algebra is $K[x]/(x^2)$.

We now define $\Qg$ for a general Brauer graph (excluding the above case, so if edge $i$ is truncated at vertex $\mu$ and the endpoints of $i$ are $\mu$ and $\nu$ then $m(\nu)\val(\nu) \geq 2$).
The vertices of $\Qg$ correspond to the edges of $\Gamma$, that is, for every edge $i \in
\Gamma_1$ there is a corresponding vertex $v_i$ in $\Qg$.
If edge $j$ is the successor
of edge $i$ at the vertex $\mu$ and edge $i$ is not a truncated edge at $\mu$ then there is an arrow from $v_i$ to $v_j$ in $\Qg$.
For each vertex $\mu$ and edge $i$
incident with $\mu$,  let $i = i_1, i_2, \ldots , i_{\val(\mu)}$ be
the edges incident with $\mu$ listed in the cyclic ordering, where
the loops are listed twice and the other edges precisely once.  We
call this the {\em successor sequence of $i$ at $\mu$}.  We set
$i_{\val(\mu)+1}=i$, noting that $i$ is the successor of
$i_{\val(\mu)}$.

In case $\Gamma$ has at least one loop, care must be taken.  In such
circumstances, for each vertex $\mu$, we choose a distinguished edge,
$i_{\mu}$, incident with $\mu$.  If $\ell$ is a loop at $\mu$,
$\ell$ occurs twice in the successor sequence of $i_{\mu}$.  We
distinguish the first and second occurrences of $\ell$ in this
sequence and view the two occurrences as two edges in $\Gamma_1$.
Thus, $\Gamma_1$ is the set of all edges with the proviso that loops
are listed twice and have different successors.
In particular, if $\Gamma$ is
$\xymatrix@1{
\mu \ar@{-}@(dr,ur)[]_{i}
}$
then, since $i$ is viewed as two edges, say $i$ and $\hat i$, the successor
sequence of $i$ at vertex $\mu$ is $i,\hat i$ and the successor sequence of $\hat i$ at $\mu$ is $\hat i,i$.  These sequences imply that there are two arrows in the quiver
$\Qg$,
and $\Qg  =
\hskip .6cm
\xymatrix@1{
\bullet \ar@(dl,ul)[]\ar@(dr,ur)[]_{\phantom{i}}
}.$

In order to define the Brauer graph relations $\rho_\Gamma$ we need
a {\em quantizing function} $q$. Let $\X_\Gamma = \{(i, \mu) \mid i
\in \Gamma_1 \mbox{ is incident with $\mu \in \Gamma_0$ and $i$ is
not truncated at either of its endpoints}\}$ and let $q\colon
\X_\Gamma \to K\setminus\{0\}$ be a set function. We denote $q((i,
\mu))$ by $q_{i, \mu}$. With this additional data we call $(\Gamma,
\mo, m, q)$ a {\em quantized Brauer graph}.
We remark that if the Brauer graph $\G$ is $\xymatrix{ \mu \ar@{-}[r] & \nu }$
then $\X_{\G} = \emptyset$. Furthermore, if the Brauer graph algebra is assumed to
be symmetric and if the field is algebraically closed, then we can set $q=1$.

 There are three types of relations for $(\Gamma, \mo, m, q)$. Note that
write our paths from left to right.

\textit{Relations of type one.} For each vertex $\mu$ and edge
$i$ incident with $\mu$, which is not truncated at the vertex $\mu$,
let $i = i_1, i_2, \ldots , i_{\val(\mu)}$ be the successor sequence
of $i$ at $\mu$.  From this we obtain a cycle $C_{i, \mu} = a_1a_2
\cdots a_{\val(\mu)}$ in $\Qg$ where the arrow $a_r$ corresponds to
the edge $i_{r+1}$ being the successor of the edge $i_r$ at the
vertex $\mu$. With this notation, for each edge $i \in \Gamma$ with
endpoints $\mu$ and $\nu$ so that $i$ is not truncated at either
$\mu$ or $\nu$, $\rho_\Gamma$ contains either $q_{i, \mu} C_{i,
\mu}^{m(\mu)} - q_{i, \nu} C_{i, \nu}^{m(\nu)}$ or $q_{i, \nu} C_{i,
\nu}^{m(\nu)} - q_{i, \mu} C_{i, \mu}^{m(\mu)}$. We call this a type
one relation. Note that since one of these relations is the negative
of the other, the ideal $\Ig$ does not depend on this choice.

\textit{Relations of type two.} The
second type of relations occurs if $i$ is a truncated edge at the
vertex $\mu$ and the endpoints of $i$ are $\mu$ and $\nu$. Let
$C_{i, \nu} = b_1b_2 \cdots b_{\val(\nu)}$ be the cycle associated
to edge $i$ incident with vertex $\nu$. In this case we have a
relation $C_{i, \nu}^{m(\nu)}b_1$.

\textit{Relations of type three.} The third type of relations are
quadratic monomial relations of the form $ab$ in $K\Qg$ where $ab$
is not a subpath of any $C_{i, \mu}$.

\begin{ex}
\begin{enumerate}
\item The graph $(\G,\mo,m)$
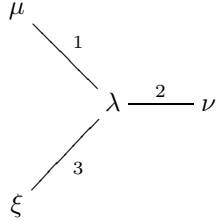
$$\xymatrix{
\mu\ar@{-}[dr]^1 & & \\
& \lambda\ar@{-}[r]^2 & \nu\\
\xi\ar@{-}[ur]_3 & &
}$$
with $m(\lambda) = m(\mu) = m(\nu) = 1$ and $m(\xi) = 2$
has edge 1 truncated at vertex $\mu$ and edge 2 truncated at vertex $\nu$.
Then $\X_{\G} = \{(3, \lambda), (3, \xi)\}$. Let $q\colon\X_{\G} \to K\setminus\{0\}$ be the quantizing function.
The Brauer graph algebra associated to the quantized Brauer graph $(\Gamma, \mo, m, q)$ has quiver
$$
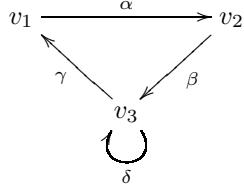
\xymatrix{
v_1\ar[rr]^\alpha & & v_2\ar[dl]^\beta\\
& v_3\ar[ul]^\gamma \ar@(dr,dl)[]^\delta &
}$$
and $\rho_\Gamma = \{q_{3,\xi}\delta^2 - q_{3,\lambda}\gamma\alpha\beta, \alpha\beta\gamma\alpha, \beta\gamma\alpha\beta, \beta\delta, \delta\gamma\}$. \\
\item The graph $(\G,\mo,m)$
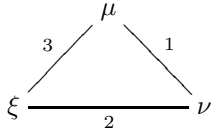
$$\xymatrix{
& \mu\ar@{-}[dr]^1\ar@{-}[dl]_3 & \\
\xi\ar@{-}[rr]_2 & & \nu
}$$
with $m(\mu) = 3$ and $m(\nu) = m(\xi) = 1$
has $\X_{\G} = \{(1, \mu), (1, \nu), (2, \xi), (2, \nu), (3, \xi), (3, \mu)\}$.
Let $q\colon \X_{\G} \to K\setminus\{0\}$. The Brauer graph algebra of $(\Gamma, \mo, m, q)$ has quiver
$$
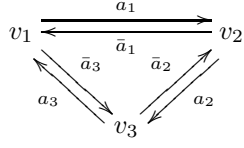
\xymatrix{
v_1\ar@<1ex>[rr]^{a_1}\ar[dr]^{\bar{a}_3} & & v_2\ar@<1ex>[dl]^{a_2}\ar[ll]^{\bar{a}_1}\\
& v_3\ar@<1ex>[ul]^{a_3}\ar[ur]^{\bar{a}_2} &
}$$
and $\rho_\Gamma = \{q_{1,\mu}(\bar{a}_3a_3)^3 - q_{1,\nu}a_1\bar{a}_1,
q_{2,\xi}a_2\bar{a}_2 - q_{2,\nu}\bar{a}_1 a_1,
q_{3,\xi}\bar{a}_2a_2 - q_{3,\mu}(a_3\bar{a}_3)^3,$\\
\hspace*{3cm}$a_1a_2, a_2a_3, a_3a_1, \bar{a}_1\bar{a}_3, \bar{a}_3\bar{a}_2, \bar{a}_2\bar{a}_1\}$.
\end{enumerate}
\end{ex}

\section{Group actions on Brauer graphs and Brauer graph algebras}\label{section:iso}

Suppose $G$ is a finite abelian group which acts on a finite connected graph $\Sigma$. Let
$\overline{\Sigma} = \Sigma/G$ be the orbit graph of $\Sigma$ under $G$, and let $\bar{\mu} = \{\mu^g \mid g \in G\}$
and $\bar{i} = \{i^g \mid g \in G\}$ where $\mu \in \Sigma_0$ and $i \in \Sigma_1$.
If the action of $G$ on $\Sigma$ is not faithful, then let $N$ be the normal subgroup of $G$
consisting of those elements that fix every vertex and every edge of $\Sigma$. Then $G/N$ acts faithfully on $\Sigma$. Without loss
of generality, we assume throughout this paper that our group actions are faithful.

Since a Brauer graph comes equipped with a cyclic ordering and multiplicity function
we now define a Brauer action of a finite abelian group on a Brauer graph as follows.

\begin{dfn}\label{dfn:Br_action}
Let $G$ be a finite abelian group and let $(\Gamma, \mathfrak{o}, m)$ be a Brauer graph.
We say that there is a {\em Brauer action} of $G$ on $(\Gamma, \mathfrak{o}, m)$ if
\begin{enumerate}
\item $G$ acts (faithfully) on the graph $\Gamma$,
\item the action of $G$ on $\G$ is orientation preserving, that is,
if $j$ is the successor of $i$ at the vertex $\mu$,
then, for all $g \in G$, $j^g$ is the successor of $i^g$ at the vertex $\mu^g$, and
\item $m(\mu) = m(\mu^g)$ for all $g \in G, \mu \in \G_0$.
\end{enumerate}
\end{dfn}

If there is a Brauer action of a finite abelian group $G$ on a
Brauer graph $(\Gamma, \mo, m)$ then $\val(\mu) = \val(\mu^g)$ for all $g \in G$ and $\mu \in \Gamma_0$, since successors are preserved by the $G$-action. In particular, by property (3), $\val(\mu)m(\mu) = \val(\mu^g)m(\mu^g)$ for all $g \in G$ and $\mu \in \Gamma_0$.

We now define a free Brauer action.

\begin{dfn}\label{dfn:free_Br_action}
Let $G$ be a finite abelian group and let $(\Gamma, \mathfrak{o}, m)$ be a Brauer graph.
A Brauer action of $G$ on $(\Gamma, \mathfrak{o}, m)$ is a {\em free Brauer action} if
$G$ acts freely on the edge set $\Gamma_1$ of $\G$, that is, if $i \in \G_1$
and $i^g = i$ then $g = \id_G$, the identity in $G$.
\end{dfn}

We remark that if there is a Brauer action of $G$ on $(\Gamma, \mo, m)$ and if there exists $g \in G$ such that $i^g = i$ for some $i \in \Gamma_1$ then, since the action of $G$ is orientation preserving and $\Gamma$ is connected, it follows that $j^g = j$ for all $j \in \Gamma_1$. Moreover, if there is a Brauer action of a non-trivial finite abelian group
$G$ on $(\G,\mo,m)$ and if  $|\Gamma_0| \geq 3$, then there is a
free Brauer action of $G$
on $(\Gamma, \mathfrak{o}, m)$ if and only if whenever $i \in \Gamma_1$ then there is some $g \in G$ with $i^g \neq i$. We emphasize that the group $G$ need not act freely on the set of vertices of $\Gamma$ for the Brauer action of $G$ on $(\Gamma, \mo, m)$ to be a free Brauer action.

The next results show that if there is a free Brauer action of $G$ on the Brauer graph $(\Gamma, \mo, m)$, then there is a multiplicity function $\om$ on $\Gb$ which is compatible with $m$ and an induced cyclic ordering $\omo$
which makes $(\Gb, \omo, \om)$ a Brauer graph.

\begin{lemma}\label{lemma:val}
Suppose there is a free Brauer action of $G$ on $(\Gamma, \mo, m)$. Then $\val(\bar{\mu}) \mid \val(\mu)$ for all $\mu \in \G_0$.
Moreover, if $i\in \Gamma_1$ is incident with $\mu$, then there exist $1 \leq k \leq \val(\mu), s \geq 0$ and $g \in G$ such that
\begin{enumerate}
\item $\val(\bar{\mu}) = k$;
\item the successor sequence of $i$ at $\mu$ is
$$i = i_1, i_2 , \ldots , i_k, i_1^g, i_2^g, \ldots , i_k^g, \ldots , i_1^{g^s}, i_2^{g^s}, \ldots , i_k^{g^s} = i_{\val(\mu)};$$
\item $\val(\mu) = \val(\bar{\mu})(s+1)$.
\end{enumerate}
\end{lemma}

\begin{proof}
Let $\mu \in \Gamma_0$ and $i \in \Gamma_1$ be incident with $\mu$. If $\val(\mu) = 1$ then clearly $\val(\bar{\mu}) = 1$ and we are done. So assume that $\val(\mu) > 1$. Consider the successor sequence of $i$ at $\mu$, $i = i_1, i_2, \ldots , i_{\val(\mu)}$. Recall that $i_{\val(\mu)+1} = i$ since $i$ is the successor of $i_{\val(\mu)}$.
There is $k$ minimal with $1 \leq k \leq \val(\mu)$ and $i_{k+1} = i^g$ for some $g \in G$. If $1 \leq \alpha < \beta \leq k$ and $i_\alpha = i_{\beta+1}^h$ for some $h \in G$ then $\alpha = 1$ and $\beta = k$. To see this, if $\alpha = 1$ then $\beta = k$ from the choice of $k$. On the other hand, if $\alpha > 1$ and there exists $h$ with $i_\alpha = i_{\beta+1}^h$, using Definition~\ref{dfn:Br_action}(2) we see that $i_1 = i_{\beta-\alpha+2}^h$. Therefore $\beta-\alpha+2 = k+1$ by definition of $k$ and hence $\beta-\alpha = k-1$. But this contradicts $\alpha \geq 2$ and $\beta \leq k$. From this it follows that $\val(\bar{\mu}) = k$.

Repeating this argument shows that the successor sequence of $i$ at $\mu$ can be written as
$$i = i_1, i_2 , \ldots , i_k, i_1^g, i_2^g, \ldots , i_k^g, \ldots , i_1^{g^s}, i_2^{g^s}, \ldots , i_k^{g^s} = i_{\val(\mu)}$$ for some $s \geq 0$.
Hence $\val(\bar{\mu})(s+1)  = \val(\mu)$ and we are done.
\end{proof}

\begin{proposition}\label{prop:mult}
Suppose there is a free Brauer action of $G$ on $(\Gamma, \mathfrak{o}, m)$.
Let $\om$ be the function
$$\om \colon \Gb_0 \rightarrow \mathbb{N} \setminus \{0\}, \ \ \ \ \bar{\mu} \mapsto \val(\mu)m(\mu)/\val(\bar{\mu}).$$ Then there is a cyclic ordering $\omo$ induced by $\mo$ such that $(\Gb, \omo, \om)$ is a Brauer graph.
\end{proposition}

\begin{proof}
We begin by showing that the cyclic ordering $\mo$ induces a cyclic ordering $\omo$. Let $\mu \in \Gamma_0$ and $i \in \Gamma_1$ be incident with $\mu$.
If $\val(\mu) = 1$ so that $\val(\bar{\mu}) = 1$, then the cyclic ordering on $\Gb$ is the only one possible.
So assume that $\val(\mu) > 1$. Consider the successor sequence of $i$ at $\mu$, $i = i_1, i_2, \ldots , i_{\val(\mu)}$. Let $k$ be as in Lemma~\ref{lemma:val}. We take the successor sequence for $\bar{i}$ at $\bar{\mu}$ to be $\bar{i}_1, \bar{i}_2, \ldots , \bar{i}_k$. This gives the desired cyclic ordering $\omo$. In particular, if $j$ is the successor of $i$ at vertex $\mu$ then $\bar{j}$ is the successor of $\bar{i}$ at vertex $\bar{\mu}$.

The set function $\om$ is well-defined by Definition~\ref{dfn:Br_action}(3) and Lemma~\ref{lemma:val}.
\end{proof}

We call the Brauer graph
$(\Gb, \omo, \om)$ in Proposition~\ref{prop:mult} the
{\em Brauer orbit graph} of $(\G,\mo,m)$ associated to the action of $G$.

\begin{proposition}\label{prop:consequences}
Suppose there is a free Brauer action of $G$ on $(\Gamma, \mathfrak{o}, m)$. Then the following properties hold.
\begin{enumerate}
\item $\om(\bar{\mu})\val(\bar{\mu}) = m(\mu)\val(\mu) = m(\mu^g)\val(\mu^g)$ for all $\mu \in \Gamma_0$ and $g \in G$.
\item Edge $i \in \Gamma_1$ is truncated at vertex $\mu$ if and only if $\bar{i} \in \Gb_1$ is truncated at vertex $\bar{\mu}$.
\item Edge $i \in \Gamma_1$ is truncated at vertex $\mu$ if and only if $i^g \in \Gamma_1$ is truncated at vertex $\mu^g$ for all $g \in G$.
\end{enumerate}
\end{proposition}

\begin{proof}
(1) is immediate from the definition of $\om$. For the proof of (2) we recall that the edge $i \in \Gamma_1$ is truncated at vertex $\mu$ if and only if
$m(\mu)\val(\mu) = 1$. Now, from (1), this holds if and only if
$\om(\bar{\mu})\val(\bar{\mu}) = 1$, which is precisely the condition that
$\bar{i} \in \Gb_1$ is truncated at vertex $\bar{\mu}$. Hence property (2) holds.
The proof of (3) is similar to that of (2), since $m(\mu)\val(\mu) = 1$ if and only if $m(\mu^g)\val(\mu^g) = 1$ from
Definition~\ref{dfn:Br_action}(3) and the subsequent remark.
\end{proof}

\begin{ex}
Let $(\G,\mo,m)$ be the Brauer graph
$$
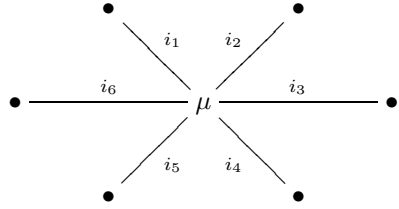
\xymatrix{
& \bullet\ar@{-}[rd]^{i_1} & & \bullet\ar@{-}[ld]_{i_2} & \\
\bullet\ar@{-}[rr]^{i_6} & & \mu\ar@{-}[rr]^{i_3} & & \bullet\\
& \bullet\ar@{-}[ur]_{i_5} & & \bullet\ar@{-}[ul]^{i_4} &
}$$
where each vertex has multiplicity 1.
Let $G$ be the cyclic group $G = {\mathbb Z}_2 = \langle g \mid g^2 = \id\rangle$ so that $G$ acts on $\Gamma$ with
$i_1^g = i_4, i_2^g = i_5, i_3^g = i_6$. Then necessarily $i_4^g = i_1, i_5^g = i_2, i_6^g = i_3$.
This is a free Brauer action of $G$ on $\Gamma$ and $\Gb$ is the Brauer graph
$$
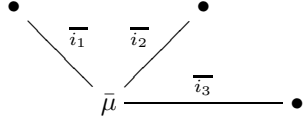
\xymatrix{
& \bullet\ar@{-}[rd]^{\overline{i_1}} & & \bullet\ar@{-}[ld]_{\overline{i_2}} & \\
& & \bar{\mu}\ar@{-}[rr]^{\overline{i_3}} & & \bullet
}$$
where $\bar{\mu}$ has multiplicity 2 and the other vertices have multiplicity 1.
\end{ex}

We now extend the concept of a Brauer action by defining a Brauer action on a quantized Brauer graph
$(\Gamma, \mo, m, q)$.

\begin{dfn}\label{dfn:quant_Br_action}
A Brauer action of a finite abelian group $G$ on the
Brauer graph $(\G,\mo,m)$ is a {\em Brauer action} of $G$ on the quantized Brauer graph $(\Gamma, \mo, m, q)$ if
$$\frac{q_{i,\mu}}{q_{i,\nu}} = \frac{q_{i^g,\mu^g}}{q_{i^g,\nu^g}},
$$
for all $g \in G$ and $\xymatrix{ \mu\ar@{-}[r]^i & \nu }$ in $\Gamma$ such that $i$ is not truncated at either of its endpoints.
\end{dfn}

To understand the importance of the quotients above, we note that the
relation of type one $q_{i,\mu}C_{i,\mu}^{m(\mu)}-q_{i,\nu}C_{i,\nu}^{m(\nu)}$
may be replaced by either
$C_{i,\mu}^{m(\mu)}-\frac{q_{i,\nu}}{q_{i,\mu}}C_{i,\nu}^{m(\nu)}$
or
$C_{i,\nu}^{m(\nu)}-\frac{q_{i,\mu}}{q_{i,\nu}}C_{i,\mu}^{m(\mu)}$
without altering the ideal $I_{\Gamma}$.   These last two relations
distinguish between the endpoints of the edge $i$.  We formalize
this as follows.  Let $\Y=\{i\in\Gamma_1\mid i \text{ is not
truncated at either of its endpoints}\}$.  Let $\E_1\colon\Y\to
\Gamma_0$ be a set function such that $\E_1(i)$ is an endpoint of
$i$.  Let $\E_2\colon\Y\to\Gamma_0$ be the set function with the
property that $\E_1(i)$ and $\E_2(i)$ are the two endpoints of $i$.
We note that $i$ is a loop if and only if $\E_1(i)=\E_2(i)$. From
the above remarks, we see that, for each $i\in \Y$, the relation of
type one may be replaced by
\[
C_{i,\E_1(i)}^{m(\E_1(i))}-\frac{q_{i,\E_2(i)}}{q_{i,\E_1(i)}}C_{i,\E_2(i)}^{m(\E_2(i))}.
\]

\begin{lemma}\label{lemma:pickE}
Suppose there is a free Brauer action of $G$ on $(\Gamma,\mathfrak{o},m,q)$.  Then there is a
choice of set functions $\E_1$ and $\E_2$ such that
\[
(\E_j(i))^g=\E_j(i^g)
\]
for all $i\in\Y, g\in G, j=1, 2$.
\end{lemma}

\begin{proof}
It suffices to prove this for $j=1$. Proposition~\ref{prop:consequences}(3)
implies that $\Y$ is a disjoint union of
orbits.  For each orbit, choose some edge $i$ in the orbit and arbitrarily select
an endpoint; define $\E_1(i)$ to be this endpoint.  For all other $i^g$
in the orbit of $i$, define $\E_1(i^g)=(\E_1(i))^g$.  This yields a
choice of $\E_1$ that satisfies the desired property.
\end{proof}

If there is a free Brauer action of $G$ on a
quantized Brauer graph $(\Gamma,\mathfrak{o},m,q)$, then we henceforth assume
that the functions $\E_1$ and $\E_2$ satisfy
$(\E_j(i))^g=\E_j(i^g)$, for all $i\in\Y, g\in G$, and $j=1, 2$.
We are now in a position to define a
quantizing function $\oq$ for the
Brauer orbit graph $(\Gb,\omo,\om)$ and thus
extend Proposition~\ref{prop:mult}. Recall that
\[ \X_{\overline{\Gamma}}=\{(\bar
i,\bar{\mu})\mid \bar{i}\in\Gb_1 \text{ is incident with }
\bar{\mu}\in\Gb_0 \text{ and } \bar{i} \text{ is not truncated at
either of its endpoints}\}. \]
Define $\oq\colon
\X_{\overline{\Gamma}}\to K\setminus\{0\}$ by
\[
\oq((\bar{i},\bar{\mu}))=
\frac{q_{i,\mu}}{q_{i,\E_1(i)}}.
\]
We denote $\oq((\bar{i},\bar{\mu}))$ by $\oq_{\bar{i},\bar{\mu}}$
and note that, by the above discussion, $\oq$ is well defined.
Summarizing, we have the following result.

\begin{proposition}\label{prop:bar-q}
Suppose there is a free Brauer action of $G$ on $(\Gamma,\mo,m,q)$.
Then there is a quantizing function $\oq$ so that $(\Gb, \omo, \om, \oq)$ is a quantized Brauer graph
satisfying the property that if $i\in\Y$ with endpoints $\mu$ and $\nu$, then
\[
\frac{q_{i,\mu}}{q_{i,\nu}}=\frac{\oq_{\bar{i},\bar{\mu}}}{\oq_{\bar{i},\bar{\nu}}}.
\]
\end{proposition}

We call the quantized Brauer graph
$(\Gb, \omo, \om,\oq)$ in Proposition~\ref{prop:bar-q} the
{\em quantized Brauer orbit graph} of $(\G,\mo,m,q)$ associated to the action of $G$.

Now we turn our attention to the Brauer graph algebra associated to a Brauer graph on which there is a free Brauer action.

\begin{lemma}\label{lemma:freeactionQg}
Suppose there is a free Brauer action of $G$ on the Brauer graph $(\Gamma, \mo, m)$. Then there is an induced free group action of $G$ on the quiver $\Qg$. Moreover if $(\Gamma, \mo, m, q)$ is a quantized Brauer graph then the induced $G$-action satisfies $$x \in \Ig \mbox{ if and only if } x^g \in \Ig \mbox{ for all } g \in G.$$
\end{lemma}

\begin{proof}
Let $v_i$ be a vertex in the quiver $\Qg$ so that $v_i$ corresponds to the
edge $i$ in $\Gamma$. We define the action of $G$ on the vertices of $\Qg$
by $(v_i)^g = v_{i^g}$ for $g \in G$. Let $a$ be an arrow from $v_i$ to $v_j$ in the
quiver $\Qg$; then $i$ and $j$ are two edges in $\G$ incident with the
same vertex and such that $j$ is the successor of $i$.
By Definition~\ref{dfn:Br_action}(2), the
edge $j^g$ is the successor of the edge $i^g$ for all $g \in G$. So, for $g \in G$, we define
$a^g$ to be the arrow from $v_{i^g}$ to $v_{j^g}$ in
$\Qg$. It is now straightforward to show, for all $g, h \in G$, that $((v_i)^g)^h =
(v_{i^g})^h = v_{(i^g)^h}= v_{i^{(gh)}} = (v_i)^{(gh)}$ and, similarly, that $(a^g)^h = a^{(gh)}$.

To show that this is a free action on $\Qg$, suppose that, for a vertex
$v_i$ in $\Qg$ (which corresponds to an edge $i$ in $\G$) and $g$ in $G$, we have
$(v_i)^g = v_i$. Then $v_{i^g} = v_i$ so by definition of $\Qg$ we have that $i^g = i$.
Hence by Definition~\ref{dfn:free_Br_action} we have $g = \id_G$. It follows that $G$ acts freely on $\Qg$.

Now suppose that we have a type one relation
$$C_{i,\E_1(i)}^{m(\E_1(i))}-\frac{q_{i,\E_2(i)}}{q_{i,\E_1(i)}}C_{i,\E_2(i)}^{m(\E_2(i))}.$$
From Definition~\ref{dfn:quant_Br_action} and Lemma~\ref{lemma:pickE} we have that
$$\left ( C_{i,\E_1(i)}^{m(\E_1(i))}-\frac{q_{i,\E_2(i)}}{q_{i,\E_1(i)}}C_{i,\E_2(i)}^{m(\E_2(i))}\right )^g =
C_{i^g,\E_1(i^g)}^{m(\E_1(i^g))}-\frac{q_{i^g,\E_2(i^g)}}{q_{i^g,\E_1(i^g)}}C_{i^g,\E_2(i^g)}^{m(\E_2(i^g))}$$
which is in $\Ig$.
It is easy to see that if $x$ is a relation of type 2 or type 3 then $x^g$ is also a relation of type 2 or type 3 respectively, for all $g \in G$. This completes the proof.
\end{proof}

Suppose there is a free Brauer action of $G$ on $(\Gamma, \mo, m)$. The induced action of $G$ on $\Qg$ given in Lemma~\ref{lemma:freeactionQg},
allows us to consider the orbit quiver $\barQg$, which is defined in the
Appendix.

\begin{proposition}\label{EqualityQuivers}
Suppose there is a free Brauer action of $G$ on $(\Gamma, \mo, m)$. Then $\barQg \cong \Qgb$ as quivers.
\end{proposition}

\begin{proof}
We prove that $(\barQg)_0 \cong (\Qgb)_0$ and $(\barQg)_1 \cong (\Qgb)_1$.

The set $(\barQg)_0$ corresponds to the orbits of vertices of $\Qg$ under the action
of $G$. The set $(\Qgb)_0$ is the set of vertices of $\Qgb$ and so corresponds to the set of edges of $\Gb$, that is, to the orbits of edges of $\Gamma$ under the action of $G$.
By definition of $\Qg$, the vertices of $\Qg$ correspond to the edges of $\Gamma$. From the definition of the action
of $G$ on $\Qg$ in Lemma~\ref{lemma:freeactionQg}, the two sets of orbits are in bijection. Thus $(\barQg)_0 \cong (\Qgb)_0$.

The set $(\Qgb)_1$ is the set of arrows of
$\Qgb$, and so corresponds to the orbits under the action of $G$ of a vertex
$\mu \in \Gamma$ with two incident edges $i$ and $j$ where $j$ is the successor of $i$.
By Definition~\ref{dfn:Br_action}(2), the edges $i^g$ and $j^g$ are
incident with the vertex $\mu^g$ and $j^g$ is the successor of
$i^g$. Thus the orbit of this action on an arrow $\xymatrix{v_i \ar[r]^{\alpha} & v_j}$ in $\Qg$
corresponds to an arrow
$\xymatrix{v_{\bar{i}} \ar[r]^{\bar{\alpha}} & v_{\bar{j}}}$ in $(\Qgb)_1$.
On the other hand, an element of $(\barQg)_1$ corresponds to an orbit $\overline{\xymatrix{v_i
\ar[r]^\alpha & v_j}}$ under the action of $G$ of an arrow $\xymatrix{v_i
\ar[r]^\alpha & v_j}$ in $\Qg$.
It is now easy to see that there is a bijection between
elements $\xymatrix{ v_{\bar{i}} \ar[r]^{\bar{\alpha}}
& v_{\bar{j}}}$ of $(\Qgb)_1$ and elements
$\overline{\xymatrix{v_i \ar[r]^\alpha & v_j}}$ of $(\barQg)_1$.
\end{proof}

Suppose there is a free Brauer action of $G$ on the quantized Brauer graph $(\Gamma, \mo, m, q)$. Then, from the Appendix and  Lemma~\ref{lemma:freeactionQg}, we have that $G$ acts freely on $\Qg$ and $\overline{\A_\Gamma} = K\overline{\Qg}/\overline{\Ig}$.

\begin{theorem}\label{thm:firstiso}
Let $G$ be a finite abelian group with a free Brauer action on the quantized Brauer graph $(\Gamma, \mo, m, q)$. Let $(\Gb, \omo, \om, \oq)$ be the associated quantized Brauer orbit graph. Let $\A_\Gamma$ be the Brauer graph algebra associated to $(\Gamma, \mo, m, q)$ and $\A_{\Gb}$ the Brauer graph algebra associated to $(\Gb, \omo, \om, \oq)$. Then
$$\overline{\A_\Gamma} \cong \A_{\Gb}.$$
\end{theorem}

\begin{proof}
From the discussion above and Proposition~\ref{EqualityQuivers}, it suffices to show that $\barIg \cong I_{\Gb}$.
Given our choice of $\E_1$ and $\E_2$ as in Lemma~\ref{lemma:pickE} we see that $\E_1$ and $\E_2$ induce set functions $\overline{\E_1}$ and $\overline{\E_2}$ for $\Gb$.

First we consider a relation $x$ of type one in $\Ig$. For this, suppose $i$ is an edge in $\Gamma_1$ which is not truncated at either of its endpoints and let $\E_1(i) = \mu$ and $\E_2(i) = \nu$. Let $i = i_1, i_2, \ldots , i_{\val(\mu)}$ (respectively $i = j_1, j_2, \ldots , j_{\val(\nu)}$) be the successor sequence for $i$ at vertex $\mu$
(respectively $i$ at vertex $\nu$) with associated cycle $C_{i,\mu} = a_1a_2\cdots a_{\val(\mu)}$
(respectively $C_{i,\nu} = b_1b_2\cdots b_{\val(\nu)}$) in $\Qg$. Thus
$x = C_{i,\mu}^{m(\mu)}-\frac{q_{i,\nu}}{q_{i,\mu}}C_{i,\nu}^{m(\nu)}$.
By Lemma~\ref{lemma:val}, there exist $1 \leq k \leq \val(\mu)$, $1 \leq l \leq \val(\nu)$, $g, h \in G$ and $s, t \geq 0$ such that the successor sequence for $i$ at $\mu$ is
$$i_1, i_2 , \ldots , i_k, i_1^g, i_2^g, \ldots , i_k^g, \ldots , i_1^{g^s}, i_2^{g^s}, \ldots , i_k^{g^s}$$
and the successor sequence for $i$ at $\nu$ is
$$j_1, j_2 , \ldots , j_l, j_1^h, j_2^h, \ldots , j_l^h, \ldots , j_1^{h^t}, j_2^{h^t}, \ldots , j_l^{h^t}.$$
It now follows that $\bar{x} \in \barIg$ is given by
$$\bar{x} = (\bar{a}_1\bar{a}_2\cdots\bar{a}_k)^{m(\mu)(s+1)} - \frac{q_{i,\nu}}{q_{i,\mu}}(\bar{b}_1\bar{b}_2\cdots\bar{b}_l)^{m(\nu)(t+1)}.$$
From Lemma~\ref{lemma:val} we have $\val(\mu) = \val(\bar{\mu})(s+1)$ and $\val(\nu) = \val(\bar{\nu})(t+1)$, and, from Proposition~\ref{prop:consequences}(1), we have
$\om(\bar{\mu})\val(\bar{\mu}) = m(\mu)\val(\mu)$. Thus $m(\mu)(s+1) = \om(\bar{\mu})$ and $m(\nu)(t+1) = \om(\bar{\nu})$. Finally by Proposition~\ref{prop:bar-q},
$$\bar{x} = (\bar{a}_1\bar{a}_2\cdots\bar{a}_k)^{\om(\bar{\mu})} - \frac{\oq_{\bar{i},\bar{\nu}}}{\oq_{\bar{i},\bar{\mu}}}(\bar{b}_1\bar{b}_2\cdots\bar{b}_l)^{\om(\bar{\nu})}$$
which is the relation of type one in $K\Qgb$ corresponding to $\bar{i}$ under the isomorphism of
Proposition~\ref{EqualityQuivers}.

The remaining relations of types two and three are monomial and a similar argument holds for these cases.
\end{proof}

\section{Weightings on Brauer graphs}\label{section:weight}

Throughout this section $(\Delta, \mo, m)$ will denote a Brauer graph and $G$ will continue to be a finite abelian group. For each $\mu \in \Delta_0$, we define $\Z_\mu$ to be the set $\Z_\mu = \{(i,j) \mid i, j \in \Delta_1, j \mbox{ is the successor of $i$ at vertex $\mu$}\}.$
Let $\Z_\Delta$ be the disjoint union
$$\Z_\Delta = \bigcup^{\bullet}_{\mu \in \Delta_0}\Z_\mu.$$

\begin{dfn}
A set function $W\colon \Z_\Delta \to G$ is called a {\em successor weighting} of the Brauer graph $(\Delta, \mo, m)$. For $\mu \in \Delta_0$ we define the {\em order of $\mu$}, denoted $\ord(\mu)$, to be the order in $G$ of the element $$\omega_\mu = \prod_{(i,j) \in \Z_\mu} W(i,j).$$
\end{dfn}

Let $W\colon \Z_\Delta \to G$ be a successor weighting of $(\Delta, \mo, m)$. The aim of this section is to construct a new Brauer graph $(\Delta_W, \mo_W, m_W)$ and a free Brauer action of $G$ on $(\Delta_W, \mo_W, m_W)$ such that the Brauer orbit graph $(\overline{\Delta_W}, \overline{\mo_W}, \overline{m_W})$ associated to $W$ is isomorphic to $(\Delta, \mo, m)$. We begin with the construction of $\Delta_W$.

Suppose $\mu$ is a vertex in $\Delta$. Let $H_\mu$ denote the subgroup of $G$ generated by the element $\omega_\mu$, and label the cosets of $H_\mu$ in $G$ as $H_{\mu, 1}, H_{\mu, 2}, \ldots , H_{\mu, |G|/\ord(\mu)}$. We define an equivalence relation on the set $\{(i, H_{\mu, s}) \mid i \in \Delta_1 \mbox{ is incident with } \mu, 1 \leq s \leq |G|/\ord(\mu)\}$ to be the equivalence relation generated by
$$(i, H_{\mu, s}) \sim (j, H_{\mu, t}) \mbox{ if $j$ is the successor of $i$ at vertex $\mu$ and
$H_{\mu, t} = H_{\mu, s}W(i,j)$.}$$
Let $\D_\mu$ be the set of equivalence classes under this equivalence relation; we denote the equivalence class of
$(i, H_{\mu, s})$ by $[i, H_{\mu, s}]$.

We recall that if an edge $i$ is a loop at the vertex $\mu$, then,
by the remarks in Section~\ref{section:notation}, $i$ is viewed as
two edges, say $i$ and $\hat i$, having different successors.  It
can easily occur that $(i, H_{\mu, s})$ is not equivalent to $(\hat
i, H_{\mu, s})$, for all $s$.

Suppose
$i \in \Delta_1$ is incident with vertex $\mu$ in $\Delta$ and let
$i = i_1, i_2, \ldots , i_{\val(\mu)}$ be the successor sequence of
$i$ at vertex $\mu$ with our usual convention that $i_{\val(\mu)+1}
= i_1$. Define $\omega(i_1, i_r) = W(i_1, i_2)W(i_2, i_3) \cdots
W(i_{r-1}, i_r)$ for $1 \leq r \leq \val(\mu)+1$. Then $\omega_\mu =
\omega(i_1, i_{\val(\mu)+1})$.

\begin{lemma}\label{lemma:omega}
Let $W\colon \Z_\Delta \to G$ be a successor weighting of $(\Delta, \mo, m)$, $i \in \Delta_1$ be incident with $\mu \in \Delta_0$ and $i = i_1, i_2, \ldots , i_{\val(\mu)}$ be the successor sequence of $i$ at $\mu$. Let $H_{\mu,s}$ be a coset of $H_\mu$ in $G$ and $j$ an edge in $\Delta$. Then $[i, H_{\mu, s}] = [j, H_{\nu, t}]$ if and only if $\nu = \mu$, $j = i_r$, and $H_{\mu, t} = H_{\mu, s}\omega(i_1, i_r)\omega_\mu^\theta$ for some $1 \leq r \leq \val(\mu)$ and $0 \leq \theta < \ord(\mu)$.
\end{lemma}

\begin{proof}
Suppose first that $[i, H_{\mu, s}] = [j, H_{\nu, t}]$. By the definition of the equivalence relation, $\mu = \nu$, so that $j$ is incident with $\mu$ in $\Delta$. Hence $j = i_r$ for some $1 \leq r \leq \val(\mu)$. Now $(i_1, H_{\mu,s})$ is equivalent to $(i_2, H_{\mu,s}W(i_1, i_2))$, so inductively $(i_1, H_{\mu,s})$ is equivalent to $(i_r, H_{\mu,s}W(i_1, i_2)\cdots W(i_{r-1}, i_r)) = (i_r, H_{\mu,s}\omega(i_1, i_r))$. However $H_\mu = \langle\omega_\mu\rangle$ so $H_{\mu, s} = H_{\mu, s}\omega_\mu^\theta$ for all $0 \leq \theta < \ord(\mu)$. Thus
$H_{\mu, t} = H_{\mu, s}\omega(i_1, i_r)\omega_\mu^\theta$ for some $0 \leq \theta < \ord(\mu)$.

Conversely, suppose that $j=i_r$ and $H_{\mu, t} = H_{\mu, s}\omega(i_1, i_r)\omega_\mu^\theta$ for some $1 \leq r \leq \val(\mu), 0 \leq \theta < \ord(\mu)$. A similar argument to that given above shows that $(i, H_{\mu, s})$ is equivalent to $(i_r, H_{\mu, t})$ and hence $[i_1, H_{\mu, s}] = [i_r, H_{\mu, t}]$.
\end{proof}

\begin{dfn}\label{dfn:graphGammaw}
We define $\Delta_W$ to be the graph with vertex set $\{\mu_d \mid \mu \in \Delta_0, d \in \D_\mu$\} and edge set $\{i_g \mid g \in G, i \in \Delta_1\}$. If $i$ is an edge in $\Delta_1$ with endpoints $\mu$
and $\nu$ in $\Delta_0$ then the edge $i_g$ in $(\Delta_W)_1$ has endpoints
$\mu_{[i, H_{\mu, s}]}$ and $\nu_{[i, H_{\nu, t}]}$ where $H_{\mu, s}$ (respectively $H_{\nu, t}$)
is the unique coset of $H_\mu$ in $G$ (respectively $H_\nu$ in $G$) containing $g$.
\end{dfn}

If $i$ is a loop in $\Delta$ incident with vertex $\mu$, then consider $i$ as two edges $i$ and $\hat i$.
Let $g\in G$. The edge $i_g$ has endpoints  $\mu_{[i,H_{\mu,s}]}$ and
$\mu_{[\hat i,H_{\mu,s}]}$, and the edge $\hat i_g$ has endpoints  $\mu_{[\hat i,H_{\mu,s}]}$ and
$\mu_{[i,H_{\mu,s}]}$, where $H_{\mu,s}$ is the unique coset of
$H_{\mu}$ in $G$ containing $g$.

Suppose that $j$ is the successor
of $i$ at the vertex $\mu$ in $\Delta$, let $g \in G$, and let
$H_{\mu, s}$ be the coset of $H_\mu$ in $G$ containing $g$. Then
both $i_g$ and $j_{gW(i,j)}$ are incident with vertex $\mu_{[i,
H_{\mu, s}]}$ in $\Delta_W$.

\begin{proposition}\label{prop:valGammaw}
Let $W \colon \Z_\Delta \to G$ be a successor weighting of $(\Delta, \mo, m)$. For each $\mu \in \Delta_0$ and $d \in \D_\mu$ we have $\val(\mu_d) = \ord(\mu)\val(\mu)$.
\end{proposition}

\begin{proof}
Let $\mu \in \Delta_0$ and $d = [i, H_{\mu, s}] \in \D_\mu$. Suppose
that $h\in G$ and $j_h$ is incident with $\mu_d$ in $\Delta_W$.
Since one endpoint of $j_h$ is $\mu_d$ it follows from Definition~\ref{dfn:graphGammaw}
that $\mu_{[i, H_{\mu, s}]}=\mu_{[j, H_{\mu, t}]}$ where
$h\in H_{\mu,t}$. Thus $[i, H_{\mu, s}]={[j, H_{\mu, t}]}$ and, by
Lemma~\ref{lemma:omega}, $j$ is in the successor sequence
$i=i_1,i_2,\dots,i_{\val(\mu)}$ of $i$ at vertex $\mu$ in $\Delta$.
For $g'\in G$ we have $hg'\in H_{\mu,t}$ if and only if
$g'=\omega_{\mu}^{\theta}$ for some $0\le \theta<\ord(\mu)$.  Thus
there are $\val(\mu)$ choices for $j$ and $\ord(\mu)$ choices for
$h$.  Hence, $\val(\mu_d)=\ord(\mu)\val(\mu)$.
\end{proof}

In order to construct a multiplicity function $m_W$
for $\Delta_W$, we need to place an additional condition on our
successor weightings.

\begin{dfn}\label{dfn:multGammaw} \begin{enumerate}
\item A successor weighting $W\colon \Z_{\Delta}\to G$ of the Brauer graph $(\Delta, \mo, m)$ is called a
{\em Brauer weighting} if $\ord(\mu)\mid m(\mu)$ for all
$\mu\in\Delta_0$.
\item If $W\colon \Z_{\Delta}\to G$ is a Brauer weighting, we define
the function $m_W$ by
$$m_W\colon (\Delta_W)_0\to \mathbb N\setminus\{0\}, \quad \mu_d \mapsto m(\mu)/\ord(\mu).$$
\end{enumerate}
\end{dfn}

We remark that $\ord(\mu)\mid m(\mu)$ if and only if
$\omega_{\mu}^{m(\mu)}=\id_G$.

\begin{proposition}\label{prop:succ_seq}
Let $W\colon \Z_{\Delta}\to G$ be a Brauer
weighting of $(\Delta,\mo,m)$.  Suppose that $j$ is
the successor of $i$ at the vertex $\mu$ in $\Delta$,  let $g \in
G$, and let $H_{\mu, s}$ be the coset of $H_\mu$ in $G$ containing
$g$. Defining $j_{gW(i,j)}$ to be the successor of $i_g$ at vertex
$\mu_{[i, H_{\mu, s}]}$ induces a cyclic ordering $\mo_W$ so that
$(\Delta_W, \mo_W, m_W)$ is a Brauer graph.
\end{proposition}

\begin{proof}
Let $i=i_1,j=i_2,i_3,\dots, i_{\val(\mu)}$ be the successor sequence
of $i$ at the vertex $\mu$ in $\Delta$.  From
Lemma~\ref{lemma:omega} and Proposition~\ref{prop:valGammaw}, the
successor sequence of $i_g$ at vertex $\mu_{[i,H_{\mu,s}]}$ is
\[
\begin{array}{l}
(i_1)_g, (i_2)_{g\omega(i_1,i_2)}, (i_3)_{g\omega(i_1,i_3)}, \dots, (i_{\val(\mu)})_{g\omega(i_1,i_{\val(\mu)})},\\
\quad (i_1)_{g\omega_{\mu}}, (i_2)_{g\omega(i_1,i_2)\omega_{\mu}}, \dots, (i_{\val(\mu)})_{g\omega(i_1,i_{\val(\mu)})\omega_{\mu}},
\dots,\\
\quad (i_1)_{g\omega_{\mu}^{\ord(\mu)-1}}, (i_2)_{g\omega(i_1,i_2)\omega_{\mu}^{\ord(\mu)-1}}, \dots, (i_{\val(\mu)})_{g\omega(i_1,i_{\val(\mu)})\omega_{\mu}^{\ord(\mu)-1}}.
\end{array}
\]
This explicitly describes the cyclic ordering $\mo_W$.
\end{proof}

\begin{lemma}
There is a canonical group action of $G$ on $\Delta_W$ given by
$$(\mu_{[i, H_{\mu, s}]})^h = \mu_{[i, H_{\mu, s}h]} \mbox{ and } (i_g)^h = i_{gh}$$
for all $h \in G$.
\end{lemma}

\begin{proof}
We prove first that the action of $G$ on a vertex of $\Delta_W$ is well-defined. Suppose
$[i, H_{\mu, s}] = [j, H_{\mu, t}]$ and that $j$ is the successor of $i$ at the vertex $\mu$.
By definition, $H_{\mu, s}W(i,j) = H_{\mu, t}$ so that $H_{\mu, s}W(i,j)h = H_{\mu, t}h$ for all $h \in G$.
Since $G$ is abelian it follows that $H_{\mu, s}h\cdot W(i,j) = H_{\mu, t}h$ and hence
$[i, H_{\mu, s}h] = [j, H_{\mu, t}h]$. Thus the action of $G$ respects the generators of the equivalence
relation and hence the action of $G$ on $(\Delta_W)_0$ is well-defined.

We leave it to the reader to show that if $i_g$ has endpoints $\mu_{[i, H_{\mu, s}]}$ and $\nu_{[i, H_{\nu, t}]}$ then $(i_g)^h$ has endpoints $(\mu_{[i, H_{\mu, s}]})^h$ and $(\nu_{[i, H_{\nu, t}]})^h$, where $h \in G$.
\end{proof}

\begin{dfn}
A map $\varphi\colon (\Gamma, \mo, m) \to (\Gamma', \mo',m')$
is an \emph{isomorphism of Brauer graphs} if $\varphi\colon \Gamma\to \Gamma'$ is a
graph isomorphism such that
\begin{enumerate}
\item if $j$ is the successor of $i$ at vertex $\mu$ in
$\Gamma$, then  $\varphi(j)$ is the successor of $\varphi(i)$ at vertex
$\varphi(\mu)$ in $\Gamma'$, and
\item $m(\mu)=m'(\varphi(\mu))$ for all $\mu \in \Gamma_0$.
\end{enumerate}
In this case, we say that $(\Gamma, \mo, m)$ and $(\Gamma', \mo', m')$ are
\emph{isomorphic as Brauer graphs}.
\end{dfn}

It is easy to see that if
$\varphi \colon (\Gamma, \mo, m) \to (\Gamma', \mo', m')$
is an isomorphism of Brauer graphs, then $\varphi$ induces an isomorphism of the
associated quivers $\Q_{\Gamma}$ and
$\Q_{\Gamma'}$, and hence induces a $K$-algebra isomorphism of the
path algebras $K\Q_{\Gamma}$ and $K\Q_{\Gamma'}$.

\begin{theorem}\label{thm:inducedfreeBraueraction}
Let $(\Delta, \mo, m)$ be a Brauer graph, $G$ a finite abelian group, $W \colon\Z_\Delta \to G$ a Brauer weighting
and $(\Delta_W, \mo_W, m_W)$ the Brauer graph associated to $W$. Then the canonical action of $G$ on $(\Delta_W, \mo_W, m_W)$ is a free Brauer action. Moreover, if $(\overline{\Delta_W}, \overline{\mo_W}, \overline{m_W})$ is the Brauer orbit graph under this action, then $(\Delta, \mo, m)$ and $(\overline{\Delta_W}, \overline{\mo_W}, \overline{m_W})$
 are isomorphic as Brauer graphs.
\end{theorem}

\begin{proof}
We first show that the action is a free Brauer action on $(\Delta_W, \mo_W, m_W)$. Let $i \in \Delta_1$ and $g, h \in G$. Since $(i_g)^h = i_{gh}$, the action of $G$ on $\Delta_W$ is free on the edge set $(\Delta_W)_1$. If $j$ is the successor of $i$ at the vertex $\mu$ in $\Delta$ and $g \in H_{\mu,s}$, then $j_{gW(i,j)}$ is the successor of
$i_g$ at the vertex $\mu_{[i,H_{\mu,s}]}$ in $\Delta_W$ and $j_{ghW(i,j)}$ is the successor of
$i_{gh}$ at the vertex $\mu_{[i,H_{\mu,s}h]}$. Since $G$ is abelian, $ghW(i,j) = gW(i,j)h$, and we see that
$(j_{gW(i,j)})^h$ is the successor of $(i_g)^h$. Now we show that $m_W(\mu_d) = m_W(\mu_d^g)$ for all $g \in G$, where $\mu_d \in (\Delta_W)_0$. By Definition~\ref{dfn:multGammaw}, $m_W(\mu_d) = m(\mu)/\ord(\mu)$. However, $\mu_d^g = \mu_{d'}$ for some $d'\in\Z_\mu$ since $H_{\mu,s}g = H_{\mu,t}$ for some $1 \leq t \leq |G|/\ord(\mu)$. Hence $m_W(\mu_d^g) = m(\mu)/\ord(\mu)$ and $m_W(\mu_d) = m_W(\mu_d^g)$. Thus the canonical action of G on
$(\Delta_W, \mo_W, m_W)$ is a free Brauer action.

We form the Brauer orbit graph $(\overline{\Delta_W}, \overline{\mo_W}, \overline{m_W})$ under this free Brauer action. The next step is to show that $\overline{\Delta_W} \cong \Delta$ as graphs. Suppose $\mu$ is a vertex in $\Delta$ and $d \in \D_\mu$. If $\mu_d = \nu_e$ for some $\nu \in \Delta_0$ and $e \in \D_\nu$ then, from
Lemma~\ref{lemma:omega}, we have that $\mu = \nu$, so the orbit of the vertex $\mu_d \in (\Delta_W)_0$ is contained in the set $\{\mu_{d'} \mid d' \in \D_\mu\}$. Now let $d' \in \D_\mu$ and suppose $d = [i, H_{\mu,s}], d' = [j, H_{\mu,t}]$ for some $i, j \in \Delta_1$ incident with $\mu$. Let $i = i_1, i_2, \ldots , i_{\val(\mu)}$ be the successor sequence of $i$ at $\mu$. By Lemma~\ref{lemma:omega}, $j = i_r$ and $H_{\mu, t} = H_{\mu,s}\omega(i_1,i_r)\omega_\mu^\theta$ for some $1 \leq r \leq \val(\mu)$, $0 \leq \theta < \ord(\mu)$. Thus $(\mu_d)^{\omega(i_1,i_r)\omega_\mu^\theta} = (\mu_{[i_1,H_{\mu,s}]})^{\omega(i_1,i_r)\omega_\mu^\theta} =
\mu_{[i_1,H_{\mu,s}\omega(i_1,i_r)\omega_\mu^\theta]} = \mu_{[i_r,H_{\mu,t}]} = \mu_{d'}$. Hence $\mu_{d'}$ is in the orbit of $\mu_d$, and so the orbit of $\mu_d$ is $\{\mu_{d'} \mid d' \in \D_\mu\}$. Thus the vertices
$\overline{\mu_d}$ in $(\overline{\Delta_W})_0$
are in one-to one correspondence with the vertices $\mu$ in $\Delta_0$.

Now, suppose that $i$ is an edge in $\Delta$ and $g \in G$. We show that the orbit of the edge $i_g$ in $(\Delta_W)_1$ is the set $\{i_h \mid h \in G\}$. Let $j$ be an edge in $\Delta$ and $h \in G$ such that $j_h$ is in the orbit of $i_g$. Then there is some $g' \in G$ with $j_h = (i_g)^{g'} = i_{gg'}$ so that $j=i$. Since $(i_g)^{g^{-1}h} = i_h$, it follows that the orbit of $i_g$ is precisely the set $\{i_h \mid h \in G\}$. Hence
the edges $\overline{i_g}$ in $(\overline{\Delta_W})_1$
are in one-to one correspondence with the edges $i$ in $\Delta_1$.
It is now straightforward to show that there is an isomorphism of graphs $\varphi\colon\overline{\Delta_W} \to \Delta$ given by $\varphi(\overline{\mu_d}) = \mu$ and $\varphi(\overline{i_g}) = i$.

\sloppy It is clear that the induced cyclic ordering $\overline{\mo_W}$ is the cyclic ordering $\mo$ under this isomorphism $\varphi$. Finally, from Propositions~\ref{prop:consequences}(1) and \ref{prop:valGammaw}, we have that $\overline{m_W}(\overline{\mu_d})\val(\overline{\mu_d}) = m_W(\mu_d)\val(\mu_d) = m_W(\mu_d)\ord(\mu)\val(\mu)$. From
Definition~\ref{dfn:multGammaw}, $m_W(\mu_d) = m(\mu)/\ord(\mu)$ so it follows that
$\overline{m_W}(\overline{\mu_d})\val(\overline{\mu_d}) = m(\mu)\val(\mu)$. It remains to show that $\val(\overline{\mu_d}) = \val(\mu)$, for then we have that $\overline{m_W}(\overline{\mu_d}) = m(\mu)$ and hence $\overline{m_W} = m\varphi$ as required. To see this, let $\mu \in \Delta_0$, $d = [i, H_{\mu,s}]$ and $g \in G$ where $i$ is incident with $\mu$, $g \in H_{\mu,s}$ and $1 \leq s \leq |G|/\ord{\mu}$.
If the element $(i_r)_h$ is in the orbit of $(i_{r'})_{h'}$, for some $h, h' \in G$, then $r = r'$, and so, from the successor sequence of $i_g$ at $\mu_d$ given in the proof of Proposition~\ref{prop:succ_seq}, we see that
$\val(\overline{\mu_d}) \geq \val(\mu)$. However, noting that
$(i_r)_{g\omega(i_i,i_r)\omega_\mu^\theta} = \left((i_r)_{g\omega(i_i,i_r)}\right)^{\omega_\mu^\theta}$
is in the orbit of $(i_r)_{g\omega(i_i,i_r)}$,
we conclude that $\val(\overline{\mu_d}) \leq \val(\mu)$ and the proof is complete.
\end{proof}

\begin{dfn}
Let $W \colon\Z_\Delta \to G$ be a Brauer weighting of $(\Delta, \mo, m)$ and let
$(\Delta_W, \mo_W, m_W)$ be the Brauer graph associated to $W$.
We call $(\Delta_W, \mo_W, m_W)$ the {\em Brauer covering graph} of $(\Delta, \mo, m)$ associated to $W$.
\end{dfn}

\begin{ex}
Let $G$ be the cyclic group $G = {\mathbb Z}_3 = \langle g \mid g^3 = \id\rangle$ and let $\Delta$ be the graph
$$\xymatrix{
& \bullet\ar@{-}[dl]_a\ar@{-}[dr]^b & \\
 \bullet\ar@{-}[rr]_c & & \bullet
 }$$
with weighting $W$ given by
$$W(a,b) = \id , W(b,a) = g, W(b,c) = \id , W(c,b) = \id , W(c,a) = \id , W(a,c) = {\id}.$$
Following the above construction, $\Delta_W$ is
$$\xymatrix{
& \bullet\ar@{-}[rr]^{c_g}\ar@{-}[dr]^{a_g} & & \bullet\ar@{-}[dl]_{b_g} &\\
\bullet\ar@{-}[rr]^{b_{\id}} & & \bullet\ar@{-}[rr]^{a_{g^2}} & & \bullet\\
& \bullet\ar@{-}[ur]^{a_{\id}}\ar@{-}[ul]_{c_{\id}} & & \bullet\ar@{-}[ur]^{c_{g^2}}\ar@{-}[ul]_{b_{g^2}} &
}$$
Moreover $\overline{\Delta_W} = \Delta$.
\end{ex}

In the next section we investigate coverings of Brauer graph algebras associated to a Brauer covering graph.

\section{Brauer graph algebras arising from weightings on Brauer graphs}\label{section:q}

Let $(\Delta, \mo, m)$ be a Brauer graph, $(\Delta, \mo, m, q)$ a quantization of $(\Delta, \mo, m)$, and $G$ a finite abelian group. Let $W \colon\Z_\Delta \to G$ be a Brauer weighting and $(\Delta_W, \mo_W, m_W)$ the Brauer graph associated to $W$. The quantizing function $q \colon \X_\Delta \to K \setminus \{0\}$ induces a quantizing function $q_W \colon \X_{\Delta_W} \to K \setminus \{0\}$ given by $q_W((i_g, \mu_d))= q_{i, \mu}$. It is clear that this map is well-defined since if $(i_g, \mu_d) \in \X_{\Delta_W}$ then $(i, \mu) \in \X_\Delta$.
We call $(\Delta_W, \mo_W, m_W, q_W)$ the {\em quantized Brauer covering graph} (associated to $W$).

Let $W \colon \Z_\Delta \to G$ be a Brauer weighting for $(\Delta, \mo, m, q)$ and let $(\Delta_W, \mo_W, m_W, q_W)$ be the quantized Brauer covering graph associated to $W$.
The weighting $W$ induces a weight function $W^* \colon ({\mathcal Q}_\Delta)_1 \to G$ of the arrows of ${\mathcal Q}_\Delta$ as follows. If $j$ is the successor of $i$ at the vertex $\mu$ in $\Delta$ and $v_i \stackrel{a}{\longrightarrow} v_j$ is the associated arrow in $({\mathcal Q}_\Delta)_1$, then we define $W^*(a) = W(i, j)$. For a path $p=a_1 a_2 \cdots a_\sigma$ in ${\mathcal Q}_\Delta$, we define $W^*(p) = W^*(a_1)W^*(a_2) \cdots W^*(a_\sigma)$. Furthermore, $W^*$ induces a $G$-grading on $K{\mathcal Q}_\Delta$ by $p = a_1 a_2 \cdots a_\sigma$ being homogeneous of degree $W^*(p)$.

Recall that if $i \in \Delta_1$ is incident with $\mu \in \Delta_0$ with successor sequence $i = i_1, i_2, \ldots , i_{\val(\mu)}$ then $C_{i,\mu} = a_1 a_2 \cdots a_{\val(\mu)}$ in ${\mathcal Q}_\Delta$ where the arrow $a_r$ corresponds to the edge $i_{r+1}$ being the successor of the edge $i_r$ and $i_{\val(\mu)+1} = i_1$.

\begin{lemma}\label{lemma:pathshomo} If $i \in \Delta_1$ is incident with $\mu \in \Delta_0$ then $W^*(C_{i,\mu}^{m(\mu)}) = \id_G$.
\end{lemma}

\begin{proof}
\sloppy We have $W^*(C_{i,\mu})= W(i_1,i_2) W(i_2,i_3) \cdots W(i_{\val(\mu)-1}, i_{\val(\mu)}) W(i_{\val(\mu)}, i_1) = \omega_\mu$ and $\omega_\mu^{m(\mu)} = \id_G$.
\end{proof}

\begin{corollary}
The ideal $I_\Delta$ can be generated by elements of $K{\mathcal Q}_\Delta$ which are homogeneous in the $G$-grading induced by $W^*$.
\end{corollary}

\begin{proof}
The relations of type one in $I_\Delta$ are of the form $q_{i, \mu}C_{i,\mu}^{m(\mu)} - q_{i,\nu}C_{i,\nu}^{m(\nu)}$. By Lemma~\ref{lemma:pathshomo}, these are homogeneous of degree $\id_G$. The relations of types two and three are homogeneous since they are paths.
\end{proof}

The weight function $W^*: ({\mathcal Q}_\Delta)_1 \to G$ gives rise to a covering quiver ${\mathcal Q}_{W^*}$ with vertex set $({\mathcal Q}_\Delta)_0 \times G = \{ v_g \mid v \in ({\mathcal Q}_\Delta)_0, g \in G \}$ and arrow set  $({\mathcal Q}_\Delta)_1 \times G = \{a_g \mid a \in ({\mathcal Q}_\Delta)_1, g \in G \}$ such that if
$\xymatrix@1{v \ar[r]^a & w}$
is an arrow in ${\mathcal Q}_\Delta$ then
$\xymatrix@1{v_g \ar[r]^<<<<<{a_g} & w_{gW^*(a)}}$
is an arrow in ${\mathcal Q}_{W^*}$.
Define the map $\pi \colon {\mathcal Q}_{W^*} \to {\mathcal Q}_\Delta$ by $\pi(v_g)= v$ and $\pi(a_g) = a$. We
extend $\pi$ to paths $a_1 a_2 \cdots a_\sigma$ in $K{\mathcal Q}_{W^*}$ by setting $\pi(a_1 a_2 \cdots a_\sigma) = \pi(a_1) \pi(a_2) \cdots \pi(a_\sigma)$. Hence we may linearly extend $\pi$ to a map $K{\mathcal Q}_{W^*} \to K{\mathcal Q}_\Delta$ which we also denote by $\pi$. Define $I_{W^*} = \pi^{-1}(I_\Delta)$. Then $K{\mathcal Q}_{W^*}/I_{W^*}$ is the covering algebra of $\A_\Delta = K{\mathcal Q}_\Delta/I_\Delta$ and we have the following theorem, whose proof is obtained from a careful analysis of the definitions and is left to the reader.

\begin{theorem}\label{thm:w*}
\sloppy
Let $(\Delta, \mo, m, q)$ be a quantized Brauer graph, $G$ a finite abelian group and $W \colon\Z_\Delta \to G$ a Brauer weighting.  Let $(\Delta_W, \mo_W, m_W, q_W)$ be the quantized Brauer covering graph and $W^* \colon ({\mathcal Q}_\Delta)_1 \to G$ the induced weight function. Then $\A_{\Delta_W} \cong K{\mathcal Q}_{W^*}/I_{W^*}$.
\end{theorem}

The next result extends Theorem~\ref{thm:inducedfreeBraueraction} and shows that the canonical action of $G$ on $(\Delta_W, \mo_W, m_W)$ is a free Brauer action on $(\Delta_W, \mo_W, m_W, q_W)$.

\begin{proposition}
 Let $(\Delta, \mo, m, q)$ be a quantized Brauer graph and $W \colon \Z_\Delta \to G$ a Brauer weighting. Then the canonical action of $G$ on $\Delta_W$ is a free Brauer action on the quantized Brauer covering graph $(\Delta_W, \mo_W, m_W, q_W)$.
\end{proposition}

\begin{proof}
Suppose that $i \in \Delta_1$ has endpoints $\mu$ and $\nu$, $g \in G$, and the cosets $H_{\mu, s}$ and $H_{\nu, t}$ both contain $g$. Then the endpoints of $i_g$ are $\mu_d$ and $\nu_e$ where $d = [i, H_{\mu, s}]$ and $e = [i,  H_{\nu, t}]$. By Theorem~\ref{thm:inducedfreeBraueraction} there is a free Brauer action of $G$ on $(\Delta_W, \mo_W, m_W)$. It remains to show that
$$\frac{q_W((i_g, \mu_d))}{q_W((i_g, \nu_e))} = \frac{q_W((i_g, \mu_d)^h)}{q_W((i_g, \nu_e)^h)}$$
for all $h \in G$. Let $h \in G$. Then $(i_g, \mu_d)^h = ((i_g)^h, (\mu_d)^h) = (i_{gh}, \mu_{d'})$ where
$d' = [i_g, H_{\mu, s}h]$ and $(i_g, \nu_e)^h = ((i_g)^h, (\nu_e)^h) = (i_{gh}, \nu_{e'})$ where
$e' = [i_g, H_{\nu, t}h]$. It follows from the definition of $q_W$ that
$$\frac{q_W((i_g, \mu_d))}{q_W((i_g, \nu_e))} = \frac{q_{i, \mu}}{q_{i, \nu}} = \frac{q_W((i_g, \mu_d)^h)}{q_W((i_g, \nu_e)^h)}.$$
This completes the proof.
\end{proof}

\begin{dfn}
A map $\varphi\colon (\Gamma, \mo, m, q) \to (\Gamma', \mo', m', q')$
is an \emph{isomorphism of quantized Brauer graphs} if
$\varphi \colon (\Gamma, \mo, m) \to (\Gamma', \mo', m')$
is an isomorphism of Brauer graphs, and the isomorphism from
$K\Q_{\Gamma}$ to $K\Q_{\Gamma'}$ induced by $\varphi$, restricts to
an isomorphism from the ideal of relations $I_{\Gamma}$ to
$I_{\Gamma'}$.
In this case, we say that $(\Gamma, \mo, m, q)$ and $(\Gamma', \mo', m', q')$ are
\emph{isomorphic as quantized Brauer graphs}.
\end{dfn}

We are now in a position to prove the main result of this section.

\begin{theorem}\label{thm:algiso}
\sloppy
Let $(\Delta, \mo, m, q)$ be a quantized Brauer graph, $G$ a finite abelian group and $W \colon\Z_\Delta \to G$ a Brauer weighting. Let $(\Delta_W, \mo_W, m_W, q_W)$ be the quantized Brauer covering graph with the canonical action of $G$ and $(\overline{\Delta_W}, \overline{\mo_W}, \overline{m_W}, \overline{q_W})$ the associated quantized Brauer orbit graph. Then $(\Delta, \mo, m, q)$ and $(\overline{\Delta_W}, \overline{\mo_W}, \overline{m_W}, \overline{q_W})$
are isomorphic as quantized Brauer graphs and there are algebra isomorphisms
$$\A_\Delta \cong \A_{\overline{\Delta_W}} \cong \overline{\A_{\Delta_W}}.$$
\end{theorem}

\begin{proof}
\sloppy The second isomorphism holds by Theorem~\ref{thm:firstiso}. By Theorem~\ref{thm:inducedfreeBraueraction}, $(\Delta, \mo, m)$ and $(\overline{\Delta_W}, \overline{\mo_W}, \overline{m_W})$ are isomorphic as Brauer graphs. We need to show that the induced isomorphism from $K\Q_{\overline{\Delta_W}}$ to $K\Q_{\Delta}$ restricts to
an isomorphism from $I_{\overline{\Delta_W}}$ to $I_{\Delta}$. First we consider relations of type one.
Let $i$ be an edge in $\Delta$ such that $i$ is not truncated at either of its endpoints $\mu$ and $\nu$, and let $g \in G$ with $g \in H_{\mu, s}$ and $g \in H_{\nu, t}$.
By Proposition~\ref{prop:bar-q} we have
$$\frac{\overline{q_W}((\overline{i_g}, \overline{\mu_d}))}{\overline{q_W}((\overline{i_g}, \overline{\nu_e}))} = \frac{{q_W}((i_g, \mu_d))}{{q_W}((i_g, \nu_e))}$$
and, by definition of the quantizing function $q_W$, we have ${q_W}((i_g, \mu_d)) = q_{i, \mu}$ and ${q_W}((i_g, \nu_e)) = q_{i, \nu}$.
Thus
$$\frac{\overline{q_W}((\overline{i_g}, \overline{\mu_d}))}{\overline{q_W}((\overline{i_g}, \overline{\nu_e}))} = \frac{q_{i, \mu}}{q_{i, \nu}}.$$
Hence a relation of type one in $I_{\overline{\Delta_W}}$ restricts to a relation of type one in $I_{\Delta}$.
We leave it to the reader to verify the corresponding statement for relations of types two and three. Hence
$(\Delta, \mo, m, q)$ and $(\overline{\Delta_W}, \overline{\mo_W}, \overline{m_W}, \overline{q_W})$
are isomorphic as quantized Brauer graphs and
$\A_\Delta \cong \A_{\overline{\Delta_W}}$.
\end{proof}

\section{From actions to orbits to coverings}\label{section:quot-to-coverings}

Throughout this section, we assume that $G$ is a finite abelian
group with free Brauer action on the quantized Brauer graph
$(\Gamma, \mathfrak o, m, q)$.  We denote the quantized Brauer orbit
graph associated to the action of $G$ by
$(\overline{\Gamma}, \overline{\mo}, \overline{m}, \overline{q})$.
We show that there is a Brauer weighting $W\colon
\Z_{\overline \Gamma}\to G$ such that the quantized Brauer covering graph
$((\overline \Gamma)_W,(\omo)_W,(\om)_W,(\oq)_W)$ is isomorphic to $(\Gamma,\mo,m,q)$.  In particular,
the Brauer graph algebras $\A_{\Gamma}$ and $\A_{(\overline \Gamma)_W}$ are
isomorphic algebras.

We begin with the construction of $W\colon \Z_{\overline \Gamma}\to G$.  If $i$
(respectively $\mu$) is an edge (resp.\ a vertex) in $\Gamma$ then,
as before, we denote the orbit of $i$ (resp.\ $\mu$) under the action
of $G$ by $\bar i$ (resp.\ $\bar \mu$) and view $\bar i$ (resp.\ $\bar
\mu$) as both an edge (resp.\ a vertex) in $\overline \Gamma$ and as an
orbit set in $\Gamma$. For each edge $\bar i$ in $\overline \Gamma$,
we fix an edge $i_*\in\Gamma$ in the orbit $\bar i$.  Next, suppose that
$\bar \mu$ and $\bar \nu$ are the endpoints of $\bar i$ in
$\overline \Gamma$. Choose $\mu_{*}$ (respectively $\nu_*$) in the
orbit of $\bar \mu$ (resp.\ $ \bar\nu$) so that $i_*$ is incident
with $\mu_*$ and $\nu_*$. Note that there is a unique choice for
$\mu_{*}$ and $\nu_*$ unless $\bar i$ is a loop and $i_*$ is not a
loop.  Now let $\bar j$ be the successor of $\bar i$ at vertex
$\bar\mu$ in $\overline\Gamma$.  Then there is an unique edge $l \in \bar j$ such that
$l$ is the successor of $i_*$ at vertex $\mu_*$ in $\Gamma$.  Hence
$l=(j_*)^g$, for some $g\in G$.  We define $W(\bar i,\bar j) = g$.  We
note that $W$ is dependent on the choices of $i_*$ and $\mu_*$.  We
call $W$ the \emph{successor weighting associated to the
action of $G$ on $(\Gamma,\mathfrak o,m)$}.

\begin{lemma}\label{lem:assocW}
Let $W\colon\Z_{\overline\Gamma}\to
G$ be the successor weighting associated to the free Brauer action of
$G$ on $(\Gamma,\mathfrak o,m)$.  Suppose that $i$ is an edge of
$\Gamma$ incident with the vertex $\mu$ and that
$i=i_1,i_2, \ldots, i_k, i_1^g, \ldots, i_k^g, \ldots, i_1^{g^s}, \ldots, i_k^{g^s}$ is the
successor sequence of $i$ at $\mu$, where $g \in G$.  Then
\begin{enumerate}
\item $g = \omega_{\bar\mu} = W(\overline{i_1},\overline{i_2})W(\overline{i_2},\overline{i_3})\cdots W(\overline{i_k},\overline{i_{k+1}})$, where $\overline{i_{k+1}}=\overline{i_{1}}$.
\item The order of $\omega_{\bar\mu} = \ord(\bar\mu) = s+1$.
\item $\ord(\bar\mu)\mid \overline m(\bar\mu)$.
\item If $h,h'\in G$ then $\mu^h=\mu^{h'}$ if and only if
$h(h')^{-1}\in H_{\bar\mu}$, where $H_{\bar\mu}$ is the subgroup of
$G$ generated by $\omega_{\bar\mu}$.
\item The index of $H_{\bar\mu}$ in $G$ equals the number of vertices in the orbit of
$\mu$.
\end{enumerate}
\end{lemma}

\begin{proof} From our hypothesis, we see that $\overline{i_1},\overline{i_2}, \ldots, \overline{i_k}$
is the successor sequence of $\bar i$ at vertex $\bar\mu$.
From the definitions of the successor weight function $W$ and the
$\omega(\overline{i_1}, \overline{i_j})$, we conclude that the successor sequence of $i_* = (i_1)_*$ at $\mu_*$
is
$$(i_1)_*,(i_2)_*^{\omega(\overline{i_1}, \overline{i_2})}, \ldots,
(i_k)_*^{\omega(\overline{i_1}, \overline{i_k})}, (i_1)_*^{\omega_{\bar\mu}}, \ldots,
(i_k)_*^{\omega_{\bar\mu}}, \ldots,
(i_1)_*^{{\omega_{\bar\mu}}^s}, \ldots, (i_k)_*^{{\omega_{\bar\mu}}^s}.$$
Since the action of $G$ is free on the edges of $\Gamma$, and since the smallest
positive integer $l$ such that $((i_1)_*)^l=(i_1)_*$ is
$\omega_{\bar\mu}^{s+1}$, we conclude that the order of
$\omega_{\bar\mu}$ is $s+1$.  Similarly, since
$((i_1)_*)^g = ((i_1)_*)^{\omega_{\bar\mu}}$, we see that
$g=\omega_{\bar\mu}$ and we have shown (1) and (2) hold.

From Proposition~\ref{prop:consequences}, we have that
$\om(\bar\mu)=m(\mu)\val(\mu)/\val(\bar\mu)$, and from
Lemma~\ref{lemma:val}, we have that $\val(\mu)=\val(\bar\mu)(s+1)$.
So $\om(\bar\mu)=m(\mu)(s+1)$, and (3) now follows from (2).

Next, consider $(\mu_*)^h$, for
$h \in G$.  From the successor
sequence of $(i_1)_*$ at $\mu_*$, we have that $(\mu_*)^h = \mu_*$ if
and only if $h = \omega_{\bar\mu}^\theta$ for some $0 \leq \theta < \ord(\bar\mu)$.  Hence, if $h,h'\in
G$, then $\mu_*^h = \mu_*^{h'}$ if and only if $h(h')^{-1}\in H_{\bar\mu}$ and (4) now follows after
noting that $\mu = (\mu_*)^{h''}$ for some $h''\in G$.

That the index of $H_{\bar \mu}$ equals the number of vertices
in the orbit of $\mu$ follows from (4), and the proof is complete.
\end{proof}

As an immediate consequence of \ref{lem:assocW}(3), we have the following result.

\begin{corollary}\label{cor:assocW} The successor weighting associated to the free Brauer
action of $G$ on $(\Gamma,\mathfrak o,m)$ is a Brauer weighting.
\end{corollary}

We call the successor weighting associated to a free Brauer action of $G$ on  $(\Gamma,\mathfrak
o,m)$, the \emph{Brauer weighting
associated to the action of $G$ on $(\Gamma,\mathfrak
o,m)$}.

\begin{theorem}\label{thm:orbit2covering}
Suppose that $G$ is a finite abelian group with a free Brauer action on the Brauer graph
$(\Gamma,\mo, m)$. Let $(\overline\Gamma, \omo, \om)$ be the Brauer orbit graph and
$(\overline\Gamma_W, \omo_W, \om_W)$ the
Brauer covering graph obtained from the Brauer weighting
$W$ associated to the action of $G$ on $(\Gamma, \mo, m)$.  Then $(\Gamma, \mo, m)$ and
$(\overline\Gamma_W, \omo_W, \om_W)$ are isomorphic as Brauer graphs.
Moreover, if $(\Gamma,\mo, m,q)$ is a quantized Brauer graph with free
Brauer action by $G$, then
$(\Gamma,\mo, m, q)$ and $(\overline\Gamma_W, \omo_W, \om_W, \oq_W)$ are
isomorphic as quantized Brauer graphs.
In particular, the associated Brauer graph algebras $\A_{\Gamma}$ and $\A_{\overline{
\Gamma}_W}$ are isomorphic.
\end{theorem}

\begin{proof}
We begin by defining an isomorphism $\varphi\colon \overline\Gamma_W\to
\Gamma$ of graphs with cyclic ordering on the edges.  If $\bar i_g$ is
an edge in $\overline\Gamma_W$, we let
$\varphi(\bar i_g) = (i_*)^g$. The action of $G$ is free on the edges of
$\Gamma$ so $\varphi$ is a bijection when restricted to the edge sets.
For $d = [\bar i, H_{\bar\mu,s}] \in \D_{\bar\mu}$, we define
$\varphi(\bar\mu_d)= (\mu_*)^h$, where $h$ is an element of the coset
$H_{\bar\mu,s}$.  By Lemma~\ref{lem:assocW}(3) and (4), $\varphi(\bar\mu_d)$ is
independent of the choice of $h$, and so $\varphi$ is a bijection when
restricted to the vertex sets.

Let $\bar i$ be an edge in $\overline\Gamma$ incident with $\bar\mu$
and let $\bar i = \overline{i_1}, \overline{i_2}, \ldots, \overline{i_k}$ be the successor sequence of
$\bar i$ at vertex $\bar\mu$.  If $i\in\Gamma_1$ is in the orbit
$\bar i$ and is incident with vertex $\mu$ where $\mu\in\bar\mu$, and
$i$ has successor sequence
$i = i_1, \ldots, i_k, i_1^g, \ldots, i_k^g, \ldots, i_1^{g^s}, \ldots, i_k^{g^s}$
at vertex $\mu$, then, since $g = \omega_{\bar\mu}$ from Lemma~\ref{lem:assocW}(1), if $\bar i_h$ is in $\overline\Gamma_W$ then
$\bar i_h$ has successor sequence
$$\begin{array}{l}
\bar i_h= (\overline{i_1})_h,(\overline{i_2})_h^{\omega(\overline{i_1},\overline{i_2})}, \ldots,
(\overline{i_k})_h^{\omega(\overline{i_1},\overline{i_k})},
(\overline{i_1})_{h}^{\omega_{\bar\mu}},(\overline{i_2})_h^{\omega_{\bar\mu}\omega(\overline{i_1},\overline{i_2})},
\ldots, (\overline{i_k})_h^{\omega_{\bar\mu}\omega(\overline{i_1},\overline{i_k})}, \ldots,\\
\hspace*{2cm} (\overline{i_1})_{h}^{\omega_{\bar\mu}^s}, \ldots,
(\overline{i_k})_h^{\omega_{\bar\mu}^s\omega(\overline{i_1},\overline{i_k})}
\end{array}$$
at vertex $\bar \mu_d$, where $d = [\bar i, H_{\bar \mu,s}]$ and $h \in H_{\bar\mu,s}$.
Using Lemma~\ref{lem:assocW}(3),
the action of $G$ on the edges of $\Gamma$, and the definition of $\varphi$, we see that
the successor sequence of $\bar i_h$ at vertex $\bar\mu_d$ in
$\overline\Gamma_W$ is sent, under $\varphi$, to the successor sequence
of $i^h$ at vertex $\mu^h$ in $\Gamma$. From this we see
that $\varphi$ is a graph isomorphism that preserves the cyclic
ordering on the edges at each vertex.

It remains to show, identifying $\Gamma$ and $\overline{\Gamma}_W$
via $\varphi$, that $m=\overline{m}_W$ and that $q$ and $\overline{q}_W$ generate the same ideal of relations in the associated Brauer graph algebras.
From Definition~\ref{dfn:multGammaw}, we see that
if $\mu$ is a vertex in $\Gamma$ and $\bar\mu_d$ is a vertex in $\overline\Gamma_W$,
then $\om_W(\bar\mu_d)= \om(\bar\mu)/\ord(\bar\mu)$. By Proposition~\ref{prop:mult}, $\om(\bar\mu)=\val(\mu)m(\mu)/\val(\bar\mu)$. However, from Lemmas~\ref{lemma:val}
and \ref{lem:assocW}(2), we have that $\frac{\val(\mu)}{\val(\bar\mu)\ord(\bar\mu)}=1$.
Hence $\om_W(\bar\mu_d)= m(\mu)$.

If $\mu\in\Gamma_0$ and $i\in \Gamma_1$ is incident with $\mu$, then $i$ is truncated at $\mu$ if and only if $\bar i$ is truncated at $\bar \mu$ if and only if $\bar i_g$ is truncated at  $\bar \mu_d$, for all $g\in G$
and $d\in \D_{\bar \mu}$, since $\val(\mu)m(\mu)=\val(\bar\mu)\om(\bar\mu)=\val(\bar\mu_d)\om_W(\bar\mu_d)$.
Hence, $(i,\mu)\in\X_{\Gamma}$ if and only if $(\bar i,\bar\mu)\in \X_{\overline{\Gamma}}$ if and only if $(\bar i_g,\bar\mu_d)\in\X_{\overline{\Gamma}_W}$, for all $g\in G$ and $d\in \D_{\bar\mu}$.  Recall from Section~\ref{section:iso} that $\Y=\{i\in\Gamma_1
\mid i \text{ is not truncated at either of its endpoints}\}$, and by Lemma~\ref{lemma:pickE}, we can choose a set function $\E_1\colon
\Y\to\Gamma_0$ so that $\E_1(i)$ is an endpoint of $i\in\Gamma_1$ and
$(\E_1(i))^g=\E_1(i^g)$, for all $g\in G$.  Let $i\in \Gamma_1$ be incident with vertex $\mu\in\Gamma_0$,
$g\in G$, and $d\in \D_{\bar\mu}$.  From Section~\ref{section:q},
$\oq_W(\bar i_g,\bar\mu_d)=\oq(\bar i,\bar\mu)$ and from Section~\ref{section:iso}, $\oq(\bar i,\bar\mu)=\frac{q_{i,\mu}}{q_{i,\E_1(i)}}$. It is now immediate that there is a correspondence between relations of
type one in $K\mathcal{Q}_{\Gamma}$ and relations of type one in
$K\mathcal{Q}_{\overline{\Gamma}_W}$.  The reader may check the correspondences
for relations of types two and three.  Thus, the associated Brauer graph algebras $\A_{\Gamma}$ and $\A_{\overline{\Gamma}_W}$ are isomorphic.
\end{proof}

\section{Applications}\label{section:appl}

In this section we provide a number of applications of the theory.
These applications lead to two of the theorems announced in the
Introduction, namely the classification of the coverings of Brauer
graph algebras that are again Brauer graph algebras
(Theorem~\ref{thm:covering}), and the fact that any Brauer graph can
be covered by a tower of coverings, the topmost of which is a Brauer
graph with no loops, no multiple edges and whose multiplicity
function is identically one, that is, there are no exceptional
vertices (Theorem~\ref{thm:final}).

\subsection{}\label{subsect:converse}
Our first application classifies the coverings of Brauer graph algebras that are again Brauer graph algebras.

Let $(\Delta,\mo,m,q)$ be a quantized Brauer graph and $W\colon \Z_{\Delta}\to G$ be a Brauer weighting for some
finite abelian group $G$.  As in Section~\ref{section:q}, we let $W^*\colon (\mathcal{Q}_{\Delta})_1\to G$ be the weight
function induced by $W$.  Theorem~\ref{thm:w*} shows that the covering algebra obtained from $W^*$ is a Brauer graph algebra.
We now show the converse, that is, if the covering algebra of $\A_{\Delta}$, obtained from a weight function on the
arrows of ${\mathcal Q}_\Delta$, is a Brauer graph algebra, then it is isomorphic to a covering algebra
obtained from a Brauer weighting of $(\Delta,\mo,m,q)$.

Let $G$ be a finite abelian group and
$W^*\colon (\mathcal{Q}_{\Delta})_1\to G$ a weight function such that $I_{\Delta}$ is generated by weight homogeneous
elements.  Let $i$ be an edge in $\Delta$ which is not truncated at either of its endpoints,  $\mu$ and $\nu$. Let
$C_{i,\mu}$ and $C_{i,\nu}$ be the corresponding cycles associated to the
edge $i$.  Then $x=q_{i,\mu}C_{i,\mu}^{m(\mu)}-q_{i,\nu}C_{i,\nu}^{m(\nu)}$
is a relation of type one in  $I_{\Delta}$. We note, by the nature of the generating
relations of $I_{\Delta}$, that $x$ must be homogeneous, that is, $W^*(C_{i,\mu}^{m(\mu)})= W^*(C_{i,\nu}^{m(\nu)})$.
We also note that $W^*(C_{i,\mu})=\omega_{\mu}$.  If $\omega_{\mu}^{m(\mu)}$
is not $\id_G$, then, in the covering algebra, the liftings of $x$ to $K\mathcal Q_{W^*}$ induce relations
which are differences of paths that are not cycles.  These differences are minimal relations for $I_{W^*}$.
It follows that the covering algebra associated to $W^*$ has minimal
generating relations that are not of types one, two, or three, and hence the covering algebra associated
 to $W^*$ is not a Brauer graph algebra.

On the other hand, if $\omega_{\mu}^{m(\mu)}=\id_G$, we define the successor weighting $W\colon
\Z_{\Delta}\to G$ as follows.  If $(i,j)\in\Z_{\mu}$, that is, if $j$ is the successor
of $i$ at vertex $\mu$, then we define $W((i,j))=W^*(a)$, where $a$ is the arrow
in $\mathcal{Q}_{\Delta}$ associated to $(i,j)$.  Since $\omega_{\mu}^{m(\mu)}=\id_G$,
we see that $\ord(\mu)\mid m(\mu)$, and hence $W$ is the desired Brauer weighting.

Summarizing, we have the following result.

\begin{theorem}\label{thm:covering} Let $(\Delta,\mo,m,q)$ be a quantized Brauer graph with
associated Brauer graph algebra $\A_{\Delta}=K\Q_{\Delta}/I_{\Delta}$ and let $G$ be a finite abelian group.
\begin{enumerate}
\item If $W\colon \Z_{\Delta}\to G$ is a Brauer weighting of $(\Delta,\mo,m,q)$
then there is an associated weight function $W^*:(\Q_\Delta)_1\to G$ such that
$I_{\Delta}$ is generated by weight homogeneous elements and the covering
algebra $(\A_{\Delta})_{W^*}$ is isomorphic to the Brauer covering algebra
$\A_{\Delta_W}$ associated to $W$.  Moreover, for each edge $i$ in $\Delta$,
which is not truncated at either of its endpoints $\mu$ and $\nu$, we have
that $W^*(C_{i,\mu}^{m(\mu)})=\id_G=W^*(C_{i,\nu}^{m(\nu)})$.
\item  Suppose that $W^*:(\Q_\Delta)_1\to G$ is a weight function such
that $I_{\Delta}$ is generated by weight homogeneous elements, and, for each
edge $i$ in $\Delta$ which is not truncated at either of its endpoints $\mu$
and $\nu$, we have that  $W^*(C_{i,\mu}^{m(\mu)})=\id_G=W^*(C_{i,\nu}^{m(\nu)})$.
Then there is a Brauer weighting $W$ of $(\Delta,\mo,m,q)$ such that the covering algebra
$(\A_{\Delta})_{W^*}$ is isomorphic to the Brauer covering algebra $\A_{\Delta_W}$ associated to $W$.
\item If $W^*:(\Q_\Delta)_1\to G$ is a weight function such that $I_{\Delta}$ is
generated by weight homogeneous elements and there is an edge $i$ in $\Delta$
which is not truncated at both endpoints $\mu$ and $\nu$, such that
$W^*(C_{i,\mu}^{m(\mu)})\ne \id_G$, then the covering algebra
    $(\A_{\Delta})_{W^*}$ is not isomorphic to a Brauer graph algebra.
\end{enumerate}
\end{theorem}

\subsection{}\label{subsect:coverings}
Our next application deals with module categories.

Suppose that $(\Delta,\mo,m,q)$ is a quantized Brauer graph,
$\A_{\Delta}$ the associated Brauer graph algebra,
and $W\colon \Z_{\Delta}\to G$ a Brauer weighting for some finite
abelian group $G$. Let $\A_{{\Delta}_W}$ be the covering algebra associated to $W$ and
let $W^*\colon(\mathcal{Q}_{\Delta})_1\to G$ be the weight function
induced by $W$.  By covering theory, $W^*$ induces a $G$-grading on
$\A_{\Delta}$.   Then the following result holds; see Appendix, Theorem~\ref{thm:app-graded}.

\begin{theorem}\label{thm:graded-modules} Keeping the above notation, the category of $G$-graded
$\A_{\Delta}$-modules is equivalent to the category of
$\A_{{\Delta}_W}$-modules.
\end{theorem}

Furthermore, if $S$ is a simple $\A_{\Delta}$-module then $S$ is gradable.
If $\tilde S$ is a graded simple $\A_{\Delta}$-module which forgets to
$S$, then a minimal graded projective $\A_{\Delta}$-resolution of $\tilde S$
forgets to a  minimal projective $\A_{\Delta}$-resolution of $S$.
By the above theorem, the minimal graded projective $\A_{\Delta}$-resolutions
are precisely the minimal projective $\A_{{\Delta}_W}$-resolutions.  Consequently,
some cohomological questions related to $\A_{\Delta}$-modules
can be translated to cohomological questions concerning $\A_{\Delta_W}$-modules,
for example, see \cite[Theorem 3.2]{GHS08}.

\subsection{}\label{subsection:towers}
In this section, we consider a ``tower'' of coverings.  More
precisely, suppose that $(\Delta,\mo,m,q)$ is a quantized Brauer
graph and $\A_{\Delta}$ is the associated Brauer graph algebra. Let
$W_1\colon \Z_{\Delta}\to G_1$ be a Brauer weighting for some finite
abelian group $G_1$ and set $(\Delta_1,\mo_1,m_1,q_1)$ to be the
associated quantized Brauer covering graph associated to $W_1$.  Let
$n$ be some positive integer. \sloppy We say that
$(\Delta,\mo,m,q),(\Delta_1,\mo_1,m_1,q_1),\dots,(\Delta_n,\mo_n,m_n,q_n)$
is a \emph{tower of quantized Brauer covering graphs} if, for
$1<i\le n$, there are Brauer weightings $W_i\colon \Z_{\Delta_{i-1}}\to
G_i$ for finite abelian groups $G_i$, where
$(\Delta_i,\mo_i,m_i,q_i)$ is the quantized Brauer covering graph
associated to $W_i$. Applying Theorem \ref{thm:graded-modules} we
have the following result.

\begin{theorem}\label{prop:tower_result}
Let $n$ be a positive integer and
$(\Delta_0,\mo_0,m_0,q_0),(\Delta_1,\mo_1,m_1,q_1),\dots,(\Delta_n,\mo_n,m_n,q_n)$
be a tower of quantized Brauer covering graphs associated to Brauer
weightings $W_i\colon \Z_{\Delta_{i-1}}\to G_i$, for some finite abelian groups $G_i$. For $i=0,\dots, n$,
let $\A_{\Delta_i}$ be the Brauer graph algebra associated to
$(\Delta_i,\mo_i,m_i,q_i)$ and let $\Mod(\A_{\Delta_i})$ denote the
category of $\A_{\Delta_i}$-modules.  For $i=1, \dots, n$, let
${\mathcal F}_i\colon \Mod(\A_{\Delta_i})\to \Mod(\A_{\Delta_{i-1}})$ be the
forgetful functor and set ${\mathcal G} = {\mathcal F}_1{\mathcal F}_2\cdots {\mathcal F}_n \colon\Mod(\A_{\Delta_n})\to
\Mod(\A_{\Delta_0})$. Then the following properties hold.
\begin{enumerate}[(i)]
\item If $S$ is a simple $\A_{\Delta_0}$-module then there is a simple
$\A_{\Delta_n}$-module $T$ such that ${\mathcal G}(T)\cong S$.
Moreover, if \[\mathcal P:\cdots \to P^2\to P^1\to P^0\to T\to 0\]
is a minimal projective $\A_{\Delta_n}$-resolution of $T$ then, applying the exact functor ${\mathcal G}$ to  $\mathcal P$, gives a minimal projective $\A_{\Delta_0}$-resolution of $S$.
\item If $S'$ is a simple $\A_{\Delta_n}$-module, then ${\mathcal G}(S')$ is a simple
$\A_{\Delta_0}$-module.
\end{enumerate}
\end{theorem}

As an immediate consequence of Theorem \ref{prop:tower_result}, we have that if
$(\Delta_0,\mo_0,m_0,q_0)$, $(\Delta_1,\mo_1,m_1,q_1), \dots, (\Delta_n,\mo_n,m_n,q_n)$
is a tower of quantized Brauer covering graphs associated to Brauer
weightings $W_i\colon \Z_{\Delta_{i-1}}\to G_i$ with associated
quantized Brauer graph algebras $\A_{\Delta_i}$, then some
homological questions related to $\A_{\Delta_0}$-modules can be
translated to homological questions concerning
$\A_{\Delta_n}$-modules. In particular, this is applied in \cite{GSST} to determine the Ext algebra of a Brauer graph algebra.

\subsection{}\label{subsect:mult}
The next three applications show how to construct coverings of
Brauer graphs with specific properties.  We begin by considering multiplicities.

Suppose that $(\Delta,\mo,m)$ is a Brauer graph for which at least one
vertex $\mu$ has $m(\mu)>1$.  Let $n=\text{lcm}\{m(\mu)\mid \mu\in\Delta_0\}$ and
$\mathbb Z_n=\langle g\mid g^n=\id\rangle$ be the cyclic group of order $n$.
We define $W\colon\Z_{\Delta}\to \mathbb Z_n$ as follows. For each $\mu\in\Delta_0$,
choose one edge, $i_{\mu}$, incident with $\mu$. Let $(i,j)\in\Z_{\mu}$, and define
\[
W((i,j))=\begin{cases} g^{\frac{n}{m(\mu)}}& \text{ if }i=i_{\mu}\\
\id&\text{ otherwise}. \end{cases}
\]
We see that $W$ is a Brauer weighting since $\omega_{\mu}=g^{\frac{n}{m(\mu)}}$, and hence $\ord(\mu)= m(\mu)$.
If $(\Delta_W,\mo_W,m_W)$ is the Brauer covering graph associated to
$W$, then, for each vertex $\mu_d$ in $\Delta_W$,
$m_W(\mu_d)=m(\mu)/\ord(\mu)=1$.  It follows that $(\Delta_W,\mo_W,m_W)$
has the desired property.

Noting that if $(\Delta,\mo,m,q)$ is a quantized Brauer graph, the above arguments can be extended to the following result.

\begin{proposition}\label{prop:multiplicity} Let $(\Delta,\mo,m,q)$ be a quantized Brauer graph.
Then there is a finite abelian group $G$ and a
Brauer weighting $W\colon\Z_{\Delta}\to G$ of $(\Delta,\mo,m,q)$ so that the quantized
Brauer covering graph $(\Delta_W,\mo_W,m_W,q_W)$ satisfies $m_W(\mu)=1$, for all vertices
$\mu$ in $\Delta_W$.
\end{proposition}

\subsection{}\label{subsect:loop}
Our next application is the removal of loops.

Let $(\Delta,\mo,m)$ be a Brauer graph with $n$ loops, labeled
$\ell_1, \dots, \ell_n$. Let $p$ be an integer greater than $1$ and
let $G=\prod_{k=1}^n\mathbb Z_p$, where $\mathbb Z_p$ is the cyclic
group of order $p$ with generator $g$. For $k=1,\dots,n$, let
$z_k=(\id,\dots,\id,g,\id,\dots,\id)\in G$, where $g$ occurs in the
$k$-th component. For each vertex $\mu$ in $\Delta$, choose one
edge, $i_{\mu}$, incident with $\mu$.

We now define a Brauer weighting $W\colon\Z_{\Delta}\to G$ as
follows. Suppose that $(i,j)\in\Z_{\mu}$.  If $i$ is not a loop,
then set $W((i,j))=\id_G$.  If $i$ is the loop $\ell_k$  and the
edge $i$ is the first occurrence of $i$ in the successor sequence of
$i_{\mu}$, then set $W((i,j))=z_k$.  If $i$ is the second occurrence
of $i$ in the successor sequence of $i_{\mu}$, then set
$W((i,j))=z_k^{-1}$. It is immediate that $\omega_{\mu}=\id_G$, for
each vertex $\mu$, and hence $W$ is, in fact, a Brauer weighting.

Let $(\Delta_W,\mo_W,m_W)$ be the Brauer covering graph with respect
to $W$.  For each vertex $\mu$, since $\omega_{\mu}=\id_G$  and the
order of $G$ is $p^n$, there are $p^n$ cosets of $H_{\mu}$ in $G$, each
of which contains a single element.
Thus, for each vertex $\mu$ in $\Delta$, we have $p^n$ vertices
$\mu_{[i,H_{\mu,1}]},\dots,\mu_{[i,H_{\mu,p^n}]}$ in $\Delta_W$, where
$i$ is an edge incident with $\mu$. These vertices are independent of the
choice of $i$, see Section~\ref{section:weight}. To
show there are no loops in $\Delta_W$, consider the loop $\ell_k$ at the
vertex $\nu$ in $\Delta$. For ease of notation we set $\ell=\ell_k$.
Let $h\in G$ and consider $\ell_h$ and $\hat\ell_h$, where $\ell$ and
$\hat\ell$ are the first and second occurrences of $\ell$ in the successor
sequence of $i_{\nu}$. We show that $\ell_h$ is not a loop in
$\Delta_W$. A similar argument shows that $\hat\ell_h$ is also not a loop.

From the remark after Definition~\ref{dfn:graphGammaw},
the edge $\ell_h$ has endpoints $\nu_{[\ell,H_{\nu,s}]}$ and
$\nu_{[\hat\ell,H_{\nu,s}]}$, where $H_{\nu,s}$ is the unique coset of
$H_{\nu}$ in $G$ containing $h$.  From Lemma~\ref{lemma:omega},
$[\ell,H_{\nu,s}]=[\hat\ell,H_{\nu,s}]$ if and only if $H_{\nu,s}=
H_{\nu,s}\omega(\ell,\hat\ell)\omega_{\nu}^{\theta}$, for some $0\le
\theta<\ord(\nu)$.  But $\omega_{\nu}=\id_G$, and, from the definition
of $W$,  $\omega(\ell,\hat\ell)$ is not $\id_G$.  Since $H_{\nu,s}=\{h\}$, we see
that $\ell_h$ is not a loop in $\Delta_W$.

Note that we could have taken $G$ to be any finite product of $n$ non-trivial
abelian groups in the above construction and each $z_k$ to be an $n$-tuple
with non-identity element in the $k$-th component and identity elements in all
other components.

Summarizing, we have the following result.

\begin{proposition}\label{prop:loop} Let $(\Delta,\mo,m,q)$ be a quantized
Brauer graph.  Then there is a finite abelian group $G$ and a
Brauer weighting $W\colon\Z_{\Delta}\to G$ of $(\Delta,\mo,m,q)$ so that
the quantized Brauer covering graph $(\Delta_W,\mo_W,m_W,q_W)$ has the property
that $\Delta_W$ contains no loops.
\end{proposition}

\subsection{}\label{subsect:multiple}
Our final application is the removal of multiple edges. Let
$(\Delta_0,\mo_0,m_0,q_0)$ be a quantized Brauer graph. By
Proposition \ref{prop:loop} there is a finite abelian group $G_1$
and a Brauer weighting $W_1\colon\Z_{\Delta_0}\to G_1$ of
$(\Delta_0,\mo_0,m_0,q_0)$ so that the quantized Brauer covering
graph $(\Delta_{W_1},\mo_{W_1},m_{W_1},q_{W_1})$ has the property
that $\Delta_{W_1}$ contains no loops.

For simplicity of notation, let $(\Delta,\mo,m,q)=
(\Delta_{W_1},\mo_{W_1},m_{W_1},q_{W_1})$. Thus $\Delta$ contains no
loops. We say the pair of vertices $\{\mu,\nu\}$ is {\em
$\alpha$-marked} if $\alpha\ge 2$ and there are precisely $\alpha$
edges with endpoints $\mu$ and $\nu$.  List the $\alpha$-marked
pairs $\{\mu_1,\nu_1\},\dots,\{\mu_n,\nu_n\}$, where
$\{\mu_k,\nu_k\}$ is $\alpha_k$-marked.  Note that a vertex in
$\Delta$ can occur in more than one $\alpha$-marked pair. Let $G$ be
the product $G=\prod_{k=1}^n\mathbb Z_{\alpha_k}$, where $\mathbb
Z_{\alpha_k}$ is the cyclic group of order $\alpha_k$ with generator
$g_k$. For $k=1,\dots,n$, let
$z_k=(\id,\dots,\id,g_k,\id,\dots,\id)$, where $g_k$ occurs in the
$k$-th component. For each vertex $\mu$ in $\Delta$, choose one
edge, $i_{\mu}$, incident with $\mu$.  For each $\alpha$-marked pair
$\{\mu,\nu\}$ choose either $\mu$ or $\nu$ to be the {\em
distinguished vertex} of that pair.

We now define a successor weighting $W\colon\Z_{\Delta}\to G$ as
follows. Suppose that $(i,j)\in\Z_{\mu}$. If the endpoints of $i$
are the $\alpha$-marked pair $\{\mu_k,\nu_k\}$, and $\mu$ is the
distinguished vertex of this pair, then set $W((i,j))=z_k$. In all
other cases, set $W((i,j))=\id_G$. The reader may check that
$\omega_{\mu}=z_k^{\alpha_k}=\id_G$, for each vertex $\mu$, and hence $W$ is, in
fact, a Brauer weighting.

Let $(\Delta_W,\mo_W,m_W)$ be the Brauer covering graph associated
to $W$.  An argument similar to the one given in \ref{subsect:loop} shows
that $(\Delta_W,\mo_W,m_W)$ has the desired properties.

\begin{proposition}\label{prop:multedges} Let $(\Delta,\mo,m,q)$ be a
quantized Brauer graph such that $\Delta$ contains no loops.  Then
there is a finite abelian group $G$ and a
Brauer weighting $W\colon\Z_{\Delta}\to G$ of $(\Delta,\mo,m,q)$ so
that the quantized Brauer covering graph $(\Delta_W,\mo_W,m_W,q_W)$ has the property that $\Delta_W$
does not contain multiple edges between any two vertices.
\end{proposition}

\subsection{}
We combine Propositions~\ref{prop:multiplicity}, \ref{prop:loop}
and \ref{prop:multedges} to obtain our final result.

\begin{theorem}\label{thm:final}
Let $(\Delta_0,\mo_0,m_0,q_0)$ be a quantized Brauer graph.  Then
there is a tower of quantized Brauer covering graphs
$(\Delta_0,\mo_0,m_0,q_0),
(\Delta_1,\mo_1,m_1,q_1),(\Delta_2,\mo_2,m_2,q_2),(\Delta_3,\mo_3,m_3,q_3)$
such that the quantized Brauer covering graph
$(\Delta_3,\mo_3,m_3,q_3)$ has the following properties:
\begin{enumerate}
\item the multiplicity function $m_3$ is identically one, 
\item the graph $\Delta_3$ has no loops, and
\item the graph $\Delta_3$ has no multiple edges.
\end{enumerate}
\end{theorem}

\begin{proof}
By Proposition \ref{prop:multiplicity}, there is a finite abelian
group $G_1$ and a Brauer weighting $W_1\colon\Z_{\Delta_{0}}\to
G_1$ such that the associated quantized Brauer covering graph
$((\Delta_0)_{W_1},(\mo_0)_{W_1},(m_0)_{W_1},(q_0)_{W_1})$ has $(m_0)_{W_1}$ identically
$1$.  Set
$(\Delta_1,\mo_1,m_1,q_1)=((\Delta_0)_{W_1},(\mo_0)_{W_1},(m_0)_{W_1},(q_0)_{W_1})$.

By Proposition \ref{prop:loop}, there is a finite abelian group
$G_2$ and a Brauer weighting $W_2\colon\Z_{\Delta_{1}}\to G_2$ such
that the associated quantized Brauer covering graph
$((\Delta_1)_{W_2},(\mo_1)_{W_2},(m_1)_{W_2},(q_1)_{W_2})$ has the property that
$(\Delta_1)_{W_2}$ contains no loops. Set
$(\Delta_2,\mo_2,m_2,q_2)=((\Delta_1)_{W_2},(\mo_1)_{W_2},(m_1)_{W_2},(q_1)_{W_2})$. 
Note
that $m_2$ is also identically $1$, for if $\mu_d$ is a vertex in $\Delta_2$ with $\mu\in \Delta_1$, then
we have $\ord(\mu)=1$ so that $m_2(\mu_d) = m_1(\mu)=1$.

Finally, by Proposition \ref{prop:multedges}, there is a finite
abelian group $G_3$ and a Brauer weighting
$W_3\colon\Z_{\Delta_{2}}\to G_3$ such that the associated quantized
Brauer covering graph $((\Delta_2)_{W_3},(\mo_2)_{W_3},(m_2)_{W_3},(q_2)_{W_3})$ has
the property that $(\Delta_2)_{W_3}$ contains no multiple edges. Set
$(\Delta_3,\mo_3,m_3,q_3)=((\Delta_2)_{W_3},(\mo_2)_{W_3},(m_2)_{W_3},(q_2)_{W_3})$. 
We then note
 that $\Delta_3$ has no loops since $\Delta_2$ has no loops, and that $m_3$ is identically $1$ for
reasons similar to $m_2$ being
identically $1$. This completes the proof.
\end{proof}

\section{Appendix}\label{section:app}

In this appendix we review covering theory for path algebras
and their quotients.  For further information and proofs
see \cite{Green83, GHS08, BongartzGabriel81, GordonGreen82}.
We allow $G$ to be any finite group whereas we assumed that
$G$ is always a finite abelian group in the
previous sections.

Let $K$ be a field and $\Q$ a finite quiver.  Let $G$ be a finite, not
necessarily abelian, group.  We begin by showing that if
$G$ acts freely on $\Q$, then there is an {\em orbit quiver}
$\oQ$ associated to this action.  Let
$G$ act freely on $\Q$. If  $x$ is either a vertex or an arrow in $\Q$,
then we denote the action of $g\in G$ on $x$ by $x^g$, and denote
the orbit of $x$ under the $G$-action by $\bar x$.
We now construct $\oQ$. The vertices of $\oQ$ are the orbits of
vertices of $\Q$ and the arrows of $\oQ$ are the orbits of arrows
of $\Q$, that is, if $a\colon v\to w$ is an arrow in $\Q_1$, then
$\bar a\colon\bar v\to \bar w$ is an arrow in $\oQ_1$.

Consider the path algebras $K\Q$ and $K\oQ$.  The action of $G$
on $\Q$ extends to an action of $G$ on the paths of $Q$ and
hence to the path algebra $K\Q$.
Let $I$ be an ideal in $K\Q$ and set $\Lambda=K\Q/I$.
Assume that $I$ satisfies $r\in I$ if and only if
$r^g\in I$, for all $g\in G$.  Let $\overline I$ denote
the set of orbits of elements of $I$ under the action
of $G$.  It is immediate that $\overline I$ is an ideal
in $K\oQ$.  Let $\oL=K\oQ/\overline I$.  We call $\oL$
the {\em orbit algebra} associated to the action of $G$ on $\L$.

We now show that $\oL$ is $G$-graded.  We start by
constructing a set function $\overline W\colon\oQ_1\to G$
which we call the {\em
weight function on $\oQ$ induced by the action of $G$}. For
each vertex $\bar v$ in $\oQ_0$, choose a vertex $v_*$ in $\Q_0$ such that $v_*\in\bar v$.
Let $\bar a\colon\bar v\to \bar w$ be an
arrow in $\oQ_1$.  Then, by the freeness of the action of
$G$ on $\Q$, there is a unique $g\in G$ and $b\in\bar a$
such that $b\colon v_*^{\id_G}\to w_*^g$ is an
arrow in $\Q_1$.  We define  $\overline W(\bar a)=g$
and remark that $\overline W$ is dependent on the choices of the $v_*$.
We see that $\overline W$ extends linearly to $K\oQ$ by
setting $\overline W(\bar v)=\id_G$,
for $\bar v\in\oQ_0$, and, if $p=\bar a_1\bar a_2\cdots \bar a_n$
is a path in $\oQ$ with
$\bar a_i\in\oQ_1$, then
$\overline W(p)= \overline W(\bar a_n)\cdots \overline W(\bar a_2)
\overline W(\bar a_1)$. Note that in Section~\ref{section:q},
the product for $\overline W(p)$ is written in reverse order.
However, the definitions coincide since $G$ is assumed to be
abelian in that section. This choice of $\overline W$ induces a
$G$-grading on $K\oQ$ by setting $(K\oQ)_g$ to be the $K$-span of
paths $p$ such that $\overline W(p)=g$.
This $G$-grading on $K\oQ$ induces a
$G$-grading on $\oL=K\oQ/\overline I$ if and only if $\overline I$ can
be generated by weight homogeneous elements, that is, by elements
each of which is in $(K\oQ)_g$, for some $g\in G$.  It remains to show
that $\overline I$ can
be generated by weight homogeneous elements.

For this, we recall that $r\in K\Q$ is said to be {\em uniform} if there are
vertices $v$ and $w$ in $\Q_0$ such that $r=vrw$.
Since every non-zero element of $K\Q$ is uniquely a sum of uniform elements,
the ideal $I$ can be generated by uniform elements.  Let $r\in I$ and
$v,w\in\Q_0$ so that $r=vrw$.  Then, since $r^g\in I$, for all $g\in G$,
there exists $h\in G$ such that $r^h=(v_*)^{\id_G}r^h(w_*)^k$,
for some $k\in G$.  Hence, if $p$ is a path occurring in $r$, then $W(\bar p)=k$.
It follows that $\bar r\in (K\oQ)_k$ is a homogeneous element,
and hence, $\overline I$ can be generated by weight homogeneous
elements.

We summarize the above discussion in the following result.

\begin{proposition}\label{prop:free-to-grade}
Let $K$ be a field and let $G$ be a finite group which acts
freely on a finite quiver $\Q$.  Assume that $I$ is an
ideal in $K\Q$ such that $r\in I$ if and only if  $r^g\in I$,
for all $g\in G$. Let $\oL=K\oQ/\overline I$ be the orbit
algebra associated to the action of $G$ on $\L=K\Q/I$.
Then there is a weight function $\overline W\colon\overline Q_1
\to G$ induced by the action of $G$ so that $\overline I$
can be generated by weight
homogeneous elements. This weight function induces a $G$-grading on $\oL$.
\end{proposition}

We now start with a weight function and construct an associated covering
algebra.  Let $W\colon\Q_1\to G$ be a weight function on
$\Q$.  We begin by defining the quiver
$\Q_W$.  Set $(\Q_W)_0=\Q_0\times G$ and $(\Q_W)_1=\Q_1\times G$. We
denote the vertices of $\Q_W$ by $v_g$ if $v\in\Q_0$ and $g\in
G$, and the arrows of $\Q_W$ by $a_g$ if $a\in\Q_1$ and $g\in G$.
The arrows in $\Q_W$ are defined as follows: if $a\colon v\to w$ is
an arrow in $\Q_1$ and $g\in G$, then $a_g\colon v_g\to w_{W(a)g}$.
There is a surjection $\pi \colon K\Q_W \to
K\Q$ induced by $\pi(v_g) = v$ and $\pi(a_g) = a$, for all
$g \in G, v \in \Q_0$ and $a \in \Q_1$.  If $I$
is generated by a set $\rho$ of weight homogeneous elements,
then let $I_W$ be the ideal in $K\Q_W$ generated by
$\pi^{-1}(\rho)$.  We call $\L_W=K\Q_W/I_W$ the {\em covering
algebra associated to $W$}.

The group $G$ acts freely on $\Q_W$ in a canonical way; namely,
if $x$ is either a vertex or an arrow in $\Q$ and $g,h\in G$, we set
$(x_g)^h=x_{gh}$. This action can be extended to an action of $G$ on
$K\Q_W$.  From the definition of $I_W$, it is clear that $r\in I_W$ if
and only $r^g\in I_W$, for all $g\in G$.  It follows that the action of
$G$ on $K\Q_W$ induces an action of $G$ on $\L_W=K\Q_W/I_W$.

The above discussion yields the following result.

\begin{proposition}\label{prop:grade-to-free}
Let $K$ be a field, $G$ a finite group, $\Q$ a finite quiver, and
 $W\colon\Q_1\to G$ a weight function on $\Q$.
 Suppose that $I$ is an ideal in $K\Q$ which is generated by weight
 homogeneous elements.  Then $G$ acts freely on $\Q_W$ and there
 is an induced $G$-action on $\L_W=K\Q_W/I_W$, the covering algebra
 of $\L=K\Q/I$ associated to $W$.
\end{proposition}

The next result shows that the above constructions may be considered
as inverse to one another.

\begin{theorem}\label{thm:appendix} Let $K$ be a field, $G$ a finite
group, $\Q$ a finite quiver, and $I$ an ideal in $K\Q$.
\begin{enumerate}
\item Suppose $G$ acts freely on $\Q$ and $r\in I$ if and only if
$r^g\in I$, for all $g\in G$.  Then, by Propositions~\ref{prop:free-to-grade}
and~\ref{prop:grade-to-free}, $G$ acts freely on $\oQ_W$.  There
is an isomorphism of quivers from $\Q$ to $\oQ_W$ which induces
a $K$-algebra isomorphism from $\Lambda=K\Q/I$ to $\oL_W=K\oQ_W/\overline I_W$.
\item Suppose $\L=K\Q/I$ and $W\colon\Q_1\to G$ is a
weight function on $\Q$ such that $I$ can be generated by
weight homogeneous elements. Then, by Propositions~\ref{prop:free-to-grade}
and~\ref{prop:grade-to-free},
there is an isomorphism of quivers from $\Q$ to $\overline{\Q_W}$
which induces a $K$-algebra isomorphism from $\Lambda=K\Q/I$ to
$\overline{\L_W}=K\overline{\Q_W}/\overline{I_W}$.
\end{enumerate}
\end{theorem}

We also have the following theorem, which we apply in
Theorem~\ref{thm:graded-modules}.  A proof may be found in \cite[Theorem 2.5]{GHS08}.

\begin{theorem}\label{thm:app-graded}
Let $K$ be a field, $G$ a finite
group, $\Q$ a finite quiver, and $I$ an ideal in $K\Q$.
Suppose $\L=K\Q/I$ and $W\colon\Q_1\to G$ is a weight
function on $\Q$ such that $I$ can be generated by weight homogeneous elements.
Give $\L$ the $G$-grading induced from $W$.  Then the category of $G$-graded
$\L$-modules is equivalent to the category of $\L_W=K\Q_W/I_W$-modules.
\end{theorem}

Finally, let $\L=K\Q/I$ be a $G$-graded algebra and let $P\colon K\Q\to
\L$ be the canonical surjection.  Suppose that the $G$-grading of $\L$ is
induced from a weight function. It follows that, if $x$ is either a vertex or an arrow
of $\Q$, then $P(x)$ is in some $(K\Q)_g$, for some $g\in G$; that is,
$P(x)$ is a homogeneous element of degree $g$.  Our final proposition shows that the converse
also holds.  The proof is left to the reader.

\begin{proposition}
Let $\L=K\Q/I$ be a $G$-graded algebra and let $P\colon K\Q\to
\L$ be the canonical surjection.  Suppose, for all $x\in\Q_0\cup\Q_1$,
there exists $g\in G$ such
that $P(x)$ is a homogeneous element of degree $g$.
Then the $G$-grading of $\L$ is induced by a weight function on $\Q$.
\end{proposition}

\end{document}